\documentclass[11pt]{amsart}

\usepackage{graphicx,latexsym,pinlabel,amssymb,arydshln}

\usepackage[all]{xy}
\SelectTips{cm}{}

\theoremstyle{plain}
\newtheorem{theorem}{Theorem}[section]
\newtheorem{lemma}[theorem]{Lemma}
\newtheorem{proposition}[theorem]{Proposition}

\newtheorem{corollary}[theorem]{Corollary}

\theoremstyle{definition}

\newtheorem{remark}[theorem]{Remark}

\newtheorem{definition}[theorem]{Definition}

\numberwithin{equation}{section}
\numberwithin{figure}{section}

\def \ie{i.e$.$, }

\def \mod{\ \operatorname{mod}\ }

\def \Z{\mathbb{Z}}
\def \Q{\mathbb{Q}}

\def \Lie{\mathfrak{L}}

\def \Malcev{\mathfrak{m}}

\def \Fat{\mathcal{F}\!at^b}
\def \Torelli{\mathcal{I}}
\def \MFat{\widetilde{\mathcal{F}\!at^b}}
\def \Mcg{\mathcal{M}}
\def \Pt{\mathfrak{Pt}}

\def \MalcevCompletion{\mathsf{M}}

\def \Aut{\operatorname{Aut}}
\def \Bar{\operatorname{B}}
\def \bch{\operatorname{bch}}

\def \F{\operatorname{F}}
\def \GLike{\operatorname{GLike}}

\def \Hom{\operatorname{Hom}}

\def \K{\operatorname{K}}
\def \Ker{\operatorname{Ker}}

\def \Img{\operatorname{Im}}

\def \P{\operatorname{P}}
\def \Prim{\operatorname{Prim}}
\def \Sp{\operatorname{Sp}}
\def \SW{\operatorname{SW}}

\def \Tor{\operatorname{Tor}}

\def \U{\operatorname{U}}
\def \Uhat{\widehat{\operatorname{U}}}

\newcommand{\gp}[1]{\mathsf{#1}}

\begin{document}

\title[]{Canonical Extensions of Morita homomorphisms\\ to the Ptolemy groupoid}

\date{June 27, 2011}

\author[]{Gw\'ena\"el Massuyeau}
\address{Institut de Recherche Math\'ematique Avanc\'ee, Universit\'e de Strasbourg \& CNRS,
7 rue Ren\'e Descartes, 67084 Strasbourg, France}
\email{massuyeau@math.unistra.fr}

\begin{abstract}
Let $\Sigma$ be a compact connected oriented surface with one boundary component.  
We extend each of Johnson's and Morita's homomorphisms to the Ptolemy groupoid of $\Sigma$.
Our extensions are canonical and take values into finitely generated free abelian groups.
The constructions are based on the $3$-dimensional interpretation of the Ptolemy groupoid,
and a chain map introduced by Suslin and Wodzicki to relate 
the homology of a nilpotent group to the homology of its Malcev Lie algebra. 
\end{abstract}

\maketitle

\begin{center}
\textsc{Introduction}
\end{center}

\vspace{0.2cm}

Let $\Sigma$ be a compact connected oriented surface of genus $g$, with one boundary component.
The mapping class group $\Mcg :=\Mcg(\Sigma)$ of the bordered surface $\Sigma$
consists of self-homeomorphisms of $\Sigma$ fixing the boundary pointwise, up to isotopy.
A classical theorem by Dehn and Nielsen asserts that the action of  $\Mcg$ on the fundamental group 
$\pi:= \pi_1(\Sigma,\star)$, whose base point $\star$ is chosen on $\partial \Sigma$, is faithful. 
Let $\Mcg[k]$ be the subgroup of $\Mcg$ acting trivially on the $k$-th nilpotent quotient  $\pi/\Gamma_{k+1} \pi$,
where $\pi=\Gamma_1 \pi \supset \Gamma_2 \pi \supset \Gamma_3 \pi \supset \cdots$ denotes the lower central series of $\pi$.
Thus, one obtains a decreasing sequence of subgroups
$$
\Mcg=\Mcg[0] \supset \Mcg[1] \supset \Mcg[2] \supset \cdots
$$
which is called the \emph{Johnson filtration} of the mapping class group, 
and whose study started with Johnson's works and was then developed by Morita.
(The reader is, for instance, refered to their surveys \cite{Johnson_survey} and \cite{Morita_survey}.)

The first term $\Mcg[1]$ of this filtration is the subgroup of $\Mcg$ acting trivially in homology,
namely the \emph{Torelli group} $\Torelli := \Torelli(\Sigma)$ of the bordered surface $\Sigma$. 
In their study of the Torelli group, Johnson \cite{Johnson_abelian,Johnson_survey} and Morita \cite{Morita_abelian}
introduced two families of group homomorphims with values in some abelian groups. 
For every $k\geq 1$, the \emph{$k$-th Johnson homomorphism}
$$
\tau_k: \Mcg[k] \longrightarrow \frac{\pi}{\Gamma_2 \pi} \otimes \frac{\Gamma_{k+1} \pi}{\Gamma_{k+2} \pi}
$$
is designed to record the action of $\Mcg[k]$ on the $(k+1)$-st nilpotent quotient $\pi/\Gamma_{k+2} \pi$:
its kernel is $\Mcg[k+1]$.
For every $k\geq 1$, the  \emph{$k$-th Morita homomorphism} 
$$
M_k: \Mcg[k] \longrightarrow H_3\left(\frac{\pi}{\Gamma_{k+1}\pi};\Z\right)
$$
is stronger than $\tau_k$: its kernel is $\Mcg[2k]$ as shown by Heap \cite{Heap}.

The Ptolemy groupoid is a combinatorial object which has arisen 
from Teichm\"uller theory in Penner's work \cite{Penner_decorated,Penner_perturbative}.
In the case of a surface with one boundary component like $\Sigma$,
the objects of the Ptolemy groupoid are decorated  graphs of a certain kind (called ``trivalent bordered fatgraphs'') 
whose edges are colored with elements of $\pi$; 
its morphisms are sequences of elementary moves  between those graphs (called ``Whitehead moves'') modulo some relations \cite{Penner_bordered}. 
For any normal subgroup $N$ of $\Mcg$, the quotient of the Ptolemy groupoid by $N$
offers a combinatorial approach to the group $N$, and this applies notably to the subgroups of the Johnson filtration.
Morita and Penner considered in \cite{MP} the problem of extending the first Johnson homomorphism to the Ptolemy groupoid.
The same kind of problem was further considered for higher Johnson homomorphisms 
and other representations of  the mapping class group in \cite{BKP,ABP}.

Let us state what may be a \emph{groupoid extension} problem in general.
Let $\Gamma$ be a group, and let $K$ be a CW-complex which is a $\K(\Gamma,1) $-space
and whose fundamental group $\pi_1(K,\centerdot)$ is identified with $\Gamma$:\\

\begin{tabular}{ll}
\emph{Given} & 
$\left\{\begin{array}{l}
\hbox{a $\Gamma$-module $A$}\\
\hbox{a crossed homomorphism $\varphi:\Gamma \to A$}
\end{array}\right.$  \\[0.5cm]
\emph{find} & 
$\left\{\begin{array}{l}
\hbox{a $\pi_1^{\operatorname{cell}}(K)$-module $\widetilde{A}$ which contains $A$ }\\
\hbox{a crossed homomorphism $\widetilde{\varphi}:  \pi_1^{\operatorname{cell}}(K) \to \widetilde{A}$}
\end{array}\right.$  \\[0.5cm]
\emph{such that }
&
$\newdir{ >}{{}*!/-5pt/\dir{>}}
\SelectTips{eu}{12}%
\xymatrix{
\pi_1(K,\centerdot) \ar@{^{(}->}[d] \! \! \! \!  \! \! \! \!  & \ar@(ul,dl)[] A \ar@{^{(}->}[d]\\
\pi_1^{\operatorname{cell}}(K)   \! \! \! \!   \! \! \! \!  & \ar@(ul,dl)[] \widetilde{A}
}$
\quad \emph{and} \quad
$\newdir{ >}{{}*!/-5pt/\dir{>}}
\SelectTips{eu}{12}%
\xymatrix{
\pi_1(K,\centerdot)  \ar@{^{(}->}[d]  \ar[r]^-{\varphi} & A \ar@{^{(}->}[d]\\
\pi_1^{\operatorname{cell}}(K) \ar[r]_-{\widetilde{\varphi}}  & \widetilde{A}.
}$ 
\end{tabular}
\vspace{0.2cm}

\noindent
Here, $\pi_1^{\operatorname{cell}}(K)$ denotes the \emph{cellular fundamental groupoid} of $K$,
\ie the category whose objects are vertices of $K$  and whose morphisms are combinatorial paths in $K$ up to combinatorial homotopies.
Of course, solutions to the groupoid extension problem  always exist: 
for instance, by choosing for each vertex $v$ of $K$ a path connecting $v$ to the base vertex $\centerdot$ of $K$,
one easily constructs an extension $\widetilde{\varphi}$ of $\varphi$ with values in $\widetilde{A} := A$.
The groupoid extension problem is pertinent in the following situation:  the group $\Gamma$ is not well understood,
and $\pi_1^{\operatorname{cell}}(K)$ offers a nice combinatorial model in which to embed $\Gamma$.
One then seeks a solution  $\widetilde{\varphi}$ defined by a \emph{canonical} formula:
the simpler the formula, the better the solution $\widetilde{\varphi}$ is.
One may need to enlarge $A$ to some $\widetilde{A}$ to achieve this, but $\widetilde{A}/A$ should not be too big.
From a cohomological viewpoint, the groupoid extension problem   consists in finding a twisted $1$-cocycle 
$$
\widetilde{\varphi}: \left\{\hbox{oriented $1$-cells of $K$} \right\} \longrightarrow A
$$ 
which represents $[\varphi] \in H^1(\Gamma;A) \simeq H^1(K;A)$:
again, the $1$-cocycle $\widetilde{\varphi}$ is desired to be canonical,
which may require taking twisted coefficients in a larger abelian group $\widetilde{A}\supset A$.\\

In this paper, we extend each of Morita's homomorphisms to the Ptolemy groupoid. 
In a first formulation, we define groupoid extensions which we call ``tautological.''
Indeed, there is a nice $3$-dimensional interpretation of the Ptolemy groupoid 
(which the author learnt from Bob Penner) in terms of Pachner moves between triangulations.
Since the original  definition of Morita homomorphisms is based on the bar resolution of groups, 
it is very natural to extend them to the Ptolemy groupoid using the same simplicial approach.
Our ``tautological'' extension $\widetilde{M_k}$ of the $k$-th Morita homomorphism $M_k$ is canonical and combinatorial,
but, in contrast with $M_k$, the target of $\widetilde{M_k}$ has infinite rank.
Thus, in a second refinement, we improve this groupoid extension by decreasing its target to a finitely generated free abelian group. 
For this, we replace groups by their Malcev Lie algebras, and we use a homological construction due to Suslin and Wodzicki \cite{SW}.
More precisely, we need the functorial chain map that they introduced  to relate
the homology of a finitely-generated torsion-free nilpotent group (with $\Q$-coefficients) to the homology of its Malcev Lie algebra.
The resulting groupoid extension of $M_k$ is denoted by $\widetilde{m}_k$,
and it is called ``infinitesimal'' in order to emphasize the use of Lie algebras.

Since Johnson  homomorphisms are determined by Morita homomorphisms in an explicit way \cite{Morita_abelian},
the same constructions can be used to extend the former to the Ptolemy groupoid.
In the abelian case ($k=1$), we recover Morita and Penner's extension of the first Johnson homomorphism \cite{MP}. 
However, it does not seem easy to relate in an explicit way
our groupoid extension $\widetilde{\tau}_k$ of the $k$-th Johnson homomorphism $\tau_k$ 
to the work of Bene, Kawazumi and Penner \cite{BKP}.  
Indeed, the extension $\widetilde{\tau}_k$ is defined ``locally''
and takes values in an abelian group that depends on the $(k+1)$-st nilpotent quotient of $\pi$;
on the contrary, the extension of $\tau_k$ constructed in \cite{BKP} is defined ``globally'' and can be considered as
a $1$-cocycle with non-abelian coefficients  that only depend on $\pi/\Gamma_2 \pi$.

As by-products, extensions of Johnson/Morita homomorphisms to \emph{crossed} homomorphisms 
on the \emph{full} mapping class group are  derived.
Their definition depends on the choice of an object in the Ptolemy groupoid. 
Extensions of Johnson homomorphisms to the full mapping class group can also be found in prior works
by Morita \cite{Morita_extension,Morita_linear}, Perron \cite{Perron} and Kawazumi \cite{Kawazumi}.
Our extensions of Johnson/Morita homomorphisms are similar to those obtained by Day in \cite{Day1,Day2}.
He also used Malcev completions of groups to construct his extensions, but with the techniques of differential topology.
As was the case in the work \cite{Massuyeau} relating Morita's homomorphisms to finite-type invariants of $3$-manifolds,
we use Malcev Lie algebras in the style of Jennings \cite{Jennings} and Quillen \cite{Quillen}:
thus, our approach is purely algebraic and it avoids Lie groups. 
Consequently, our extensions of Johnson/Morita homomorphisms to the mapping class group are purely combinatorial.
Their explicit computation is subordinate to that of the Suslin--Wodzicki chain map.\\

The paper is organized in the following manner:

\tableofcontents

\noindent
\textbf{Acknowledgements.}
The author would like to thank Alex Bene and Bob Penner for useful comments on a previous version of this manuscript.
This work was partially supported by the French ANR research project ANR-08-JCJC-0114-01.

\section{A short review of the Ptolemy groupoid}

\label{sec:review}

In this section, we recall the definition of the Ptolemy groupoid for the bordered surface $\Sigma$. 
We start by reviewing the category of bordered fatgraphs,
whose combinatorics is equivalent to the complex of arc families studied by Penner in \cite{Penner_bordered}.
The category of bordered fatgraphs has been introduced by Godin \cite{Godin},
who adapted constructions and results of Harer \cite{Harer}, 
Penner \cite{Penner_decorated,Penner_perturbative}  and Igusa \cite{Igusa} to the bordered case.
Our exposition is brief: the reader is refered to the paper \cite{Godin} for precise definitions, proofs and further references.
The Ptolemy groupoid is then defined as the cellular fundamental groupoid of a CW-complex 
that realizes the category of $\pi$-marked bordered fatgraphs.

\subsection{The category of bordered fatgraphs}

\label{subsec:fatgraphs}
 
A \emph{fatgraph} is a finite graph $G$ whose vertices $v$ can be of two types:
either $v$ is univalent, or $v$ is at least trivalent and oriented 
(\ie $v$ is equipped with a cyclic ordering of its incident half-edges).
Vertices of the first kind are called \emph{external}, while vertices of the second kind are called \emph{internal}.
An edge is \emph{external} if it is incident to an external vertex, and it is \emph{internal} otherwise.

A fatgraph $G$ can be ``thickened'' to a compact oriented surface $\mathbb{G}$ with boundary.
More precisely, every internal vertex of $G$ is thickened to a disk, 
every internal edge of $G$ is thickened to a band which connects the disk(s) corresponding to its endpoint(s),
every external edge of $G$ is thickened to a band with one ``free'' end, 
and the orientation of the surface $\mathbb{G}\supset G$ should be compatible with the vertex orientations of $G$.
The fatgraph $G$ is \emph{bordered} if there is exactly one external vertex of $G$ on each boundary component of $\mathbb{G}$.
In the sequel, we shall restrict ourselves to bordered fatgraphs $G$ such that $\mathbb{G}$ is homeomorphic to $\Sigma$.
(Therefore our bordered fatgraphs will have exactly one external vertex.) 
An example is drawn on Figure \ref{fig:fatgraph} in genus $g=1$.

\begin{figure}[h]
{\labellist \small \hair 2pt
\pinlabel {$\leadsto$} at  576 171 
\endlabellist
\includegraphics[scale=0.37]{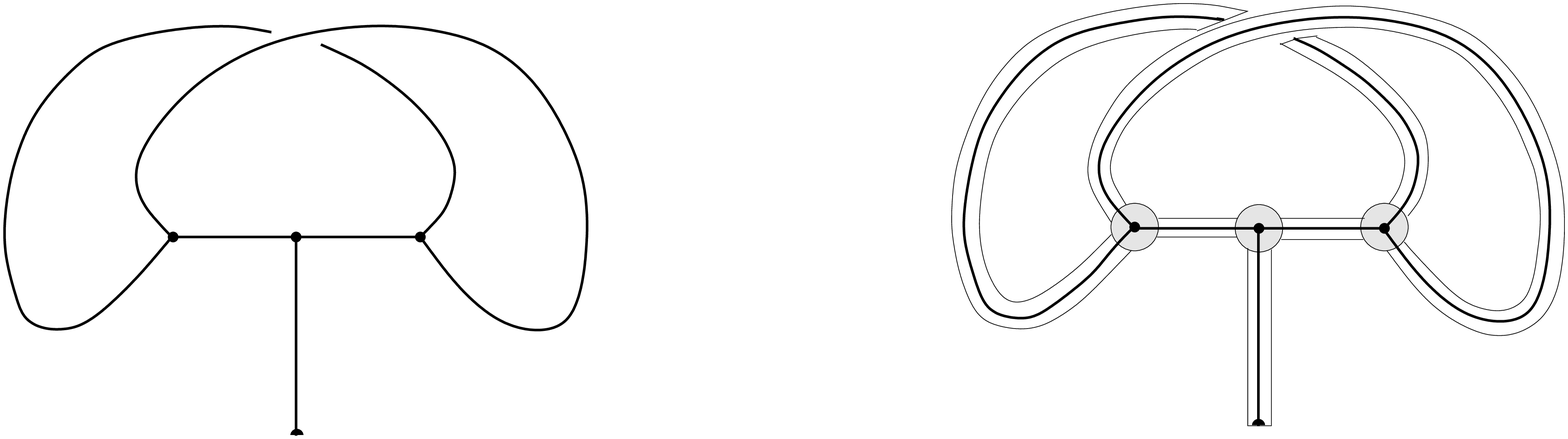}}
\caption{A fatgraph $G$ and its ``thickening'' $\mathbb{G}$.
(On pictures, we agree that the vertex orientation is counterclockwise.)}
\label{fig:fatgraph}
\end{figure}

A \emph{morphism} of bordered fatgraphs $G \to G'$ consists of two bordered fatgraphs $G$ and $G'$,
such that $G'$ is obtained from $G$ by collapsing a disjoint union of contractible subgraphs of $G$ to vertices of $G'$.
(This collapse should respect the vertex orientations and should exclude the external vertex.)
Considering fatgraphs up to isomorphisms, one obtains the category 
$$
\Fat(\Sigma) =: \Fat
$$
of bordered fatgraphs for the surface $\Sigma$. 
This category has finitely many objects and morphisms.

\begin{remark}
The category $\Fat$ only depends on the topological type of the surface $\Sigma$, \ie on its genus $g$.
This category is denoted by $\Fat_{g,1}$ in the paper \cite{Godin}, 
which applies to any compact connected oriented surface with non-empty boundary.
\end{remark}

In the sequel, we fix an arc $I$ in $\partial \Sigma$ which does not meet the base point $\star$.
Let $G$ be a bordered fatgraph and let $F \subset \partial \mathbb{G}$ be the ``free'' end of the band of $\mathbb{G}$ 
corresponding to the external edge of $G$.
A \emph{marking} of $G$ is an isotopy class of orientation-preserving embeddings 
$m:\mathbb{G} \hookrightarrow \Sigma$ such that $m(\mathbb{G})\cap \partial \Sigma= m(F)=I$.
Thus, the complement of $m(\mathbb{G})$ in $\Sigma$ is a disk whose closure is  denoted by $D_m$.
A marking $m$ of $G$ induces a map $\varpi:\{\hbox{oriented edges of $G$}\} \to \pi$.
Indeed, for every oriented edge $e$ of $G$, we can consider the band of $\mathbb{G}$ obtained by ``thickening'' $e$
and, inside this band,  there is a simple proper arc $e^*$ meeting $e$ transversely in a single point.
We orient $e^*$ so that the intersection number
$e^*\cdot e$ is $+1$  in the oriented surface $\mathbb{G}$: see Figure \ref{fig:dual_arc}.
The image $m(e^*)$ of $e^*$ is an arc in $\Sigma$ whose endpoints are on $\partial D_m$:
we connect each of them  to $\star$ by an arc in $D_m$.
The resulting loop in $\Sigma$ based at $\star$ is still denoted by $m(e^*)$
and we set  $\varpi(e):=[m(e^*)]$. Clearly, the map $\varpi$ satisfies the following:
\begin{itemize}
\item For every oriented edge $e$ of $G$, we have $\varpi(e)\cdot \varpi(\overline{e})=1$
where $\overline{e}$ denotes the edge $e$ with opposite orientation;
\item For every internal vertex $v$ of $G$, we have $\varpi(h_1^+) \cdots \varpi(h_k^+)=1$
where $h_1,\dots,h_k$ are the cyclically-ordered half-edges incident to $v$
and $h_1^+,\dots,h_k^+$ denote the corresponding edges which are oriented towards $v$;
\item The image of $\varpi$ generates the group $\pi$;
\item $\varpi$ takes the value $\zeta := [\partial \Sigma]$ on the external edge, 
which is oriented from the external vertex to the internal vertex.
\end{itemize}
Conversely, a map $\varpi$ verifying those four properties is called a \emph{$\pi$-marking} of $G$ in \cite{BKP},
where it is observed that ``markings'' and ``$\pi$-markings'' are equivalent notions.
Therefore, we shall confuse those two notions in the sequel. 

\begin{figure}[h]
\labellist \small \hair 2pt
\pinlabel {$\partial \mathbb{G}$} [r] at 1 3
\pinlabel {$\partial \mathbb{G}$} [r] at 1 219
\pinlabel {$e^*$} [r] at 380 40
\pinlabel {$e$} [b] at 570 115
\pinlabel {\Large $\circlearrowleft$} at 570 50
\pinlabel {\scriptsize $+$}  at 570 50
\endlabellist
\centering
\includegraphics[scale=0.25]{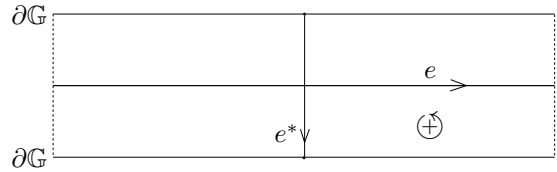}
\caption{The arc $e^*$ in $\mathbb{G}$ dual to an edge $e$ of $G$.}
\label{fig:dual_arc}
\end{figure}

A \emph{morphism} of $\pi$-marked bordered fatgraphs $G \to G'$ is a morphism of bordered fatgraphs
that respects the $\pi$-markings. (In other words, if the set of oriented edges of $G'$ is identified
with a subset of that of $G$, then the $\pi$-marking of $G'$ should be the restriction of that of $G$.) 
Considering $\pi$-marked fatgraphs up to isomorphisms, one obtains the category 
$$
\MFat(\Sigma) =: \MFat
$$
of $\pi$-marked bordered fatgraphs for the surface $\Sigma$. There is an obvious forgetful functor
from $\MFat$ to $\Fat$, which induces a simplicial map
$$
p: | \MFat | \longrightarrow | \Fat |
$$
between their geometric realizations. The mapping class group $\Mcg$ acts freely on $\MFat$ by changing the $\pi$-markings,
and the quotient $| \MFat |/\Mcg$ is isomorphic to $| \Fat |$ via $p$. 

\begin{theorem}[Penner, Harer, Igusa, Godin]
\label{th:contractibility}
The space $| \MFat |$ is contractible. Therefore $|\MFat|$ is the universal cover of $| \Fat |$, 
and $| \Fat |$ is a $\operatorname{K}(\Mcg,1)$-space.
\end{theorem}

There is, furthermore, a cell decomposition on the space $| \Fat |$. For every bordered fatgraph $G$, set
$$
d(G) := \sum_{v} \left(\hbox{valence}(v)-3\right) 
$$
where the sum ranges over all internal vertices $v$ of $G$.  The fatgraph $G$ is said to be \emph{trivalent}
if all its internal vertices are trivalent or, equivalently, if $d(G)=0$.  
If $G$ is not trivalent, then $G$ can be ``decontracted'' to a trivalent fatgraph in various ways, and we denote
$$
\overline{G} \subset | \Fat |
$$
the union of all $d(G)$-dimensional simplices of the form $H_0\to H_1 \to \cdots \to H_{d(G)}=G$ 
where $H_0$ is trivalent and each morphism $H_i \to H_{i+1}$ is a single edge contraction.

\begin{theorem}[Penner, Harer, Godin]
\label{th:cell}
There is a cell decomposition on  $| \Fat |$, 
with one $d$-dimensional cell $\overline{G}$ for each bordered fatgraph $G$ such that $d(G)=d$.
\end{theorem}

\noindent
This cell decomposition can be lifted in a unique way  to the universal cover $|\MFat|$. 
In the sequel, the corresponding CW-complex is simply denoted by $\MFat$.

More generally,  we can consider for any normal subgroup $N$ of $\Mcg$ the quotient category $\MFat/N$ by the action of $N$ \cite{MP,BKP}.
Its geometric realization $|\MFat/N|$ is,  according to Theorem \ref{th:contractibility}, 
a $\operatorname{K}(N,1)$-space and, according to Theorem \ref{th:cell}, it has a canonical cell decomposition.
The corresponding CW-complex is simply denoted by $\MFat/N$.
Note that, for any $\pi$-marked trivalent bordered fatgraph $G$, the fundamental group of $\MFat/N$ based at the vertex $\{G\}$ 
can be identified with the group $N$ in a canonical way. In the sequel, this identification will be denoted by
$$
\xymatrix{
\pi_1\left(\MFat/N,\{G\}\right) \ar@{=}[r]^-G & N.
}
$$

\subsection{The Ptolemy groupoid}

Recall that the cellular fundamental groupoid $\pi_1^{\operatorname{cell}}(K)$ 
of a CW-complex $K$ is defined by the following presentation  \cite{Higgins}: 
generators are oriented $1$-cells $e$ of $K$; relations are of the form $e \cdot \overline{e}$ (for every oriented $1$-cell $e$ of $K$)
or of the form $\partial D$ (for every oriented $2$-cell $D$ of $K$). 
Relations of the first kind are called \emph{involutivity} relations.

By definition, the \emph{Ptolemy groupoid} of the surface $\Sigma$ 
is the cellular fundamental groupoid of the CW-complex $\MFat$.
We denote it by
$$
\Pt := \pi_1^{\operatorname{cell}}\left(\MFat\right).
$$
The $0$-cells of $\MFat$ are $\pi$-marked trivalent bordered fatgraphs. 
The $1$-cells correspond to $\pi$-marked bordered fatgraphs all of whose internal vertices are trivalent, 
except for a single quadrivalent vertex.
Consequently an oriented $1$-cell of  $\MFat$ can be regarded as a \emph{Whitehead move} $W:G \to G'$ 
between two $\pi$-marked trivalent bordered fatgraphs:
$$
\labellist \small \hair 2pt
\pinlabel {$G$} [r] at 0 30
\pinlabel {$G'$} [l] at 376 30
\pinlabel {$\overset{W}{\longrightarrow}$} at 190 30
\endlabellist
\centering
\includegraphics[scale=0.3]{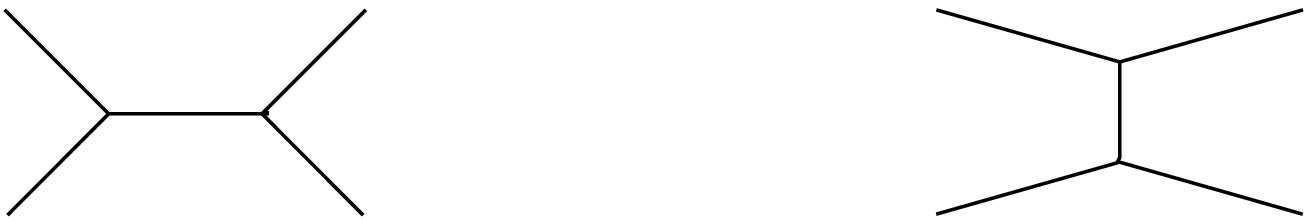}
$$
The $2$-cells of $\MFat$ correspond to $\pi$-marked bordered fatgraphs all of whose internal vertices are trivalent,
except for a single vertex of valence $5$, or, two vertices of valence $4$. 
A $2$-cell of the first kind gives a \emph{pentagon} relation between Whitehead moves, 
and a $2$-cell of the second kind gives a \emph{commutativity} relation, as shown on Figure \ref{fig:relations}.
\begin{figure}
{\labellist \small \hair 2pt
\endlabellist
\includegraphics[scale=0.25]{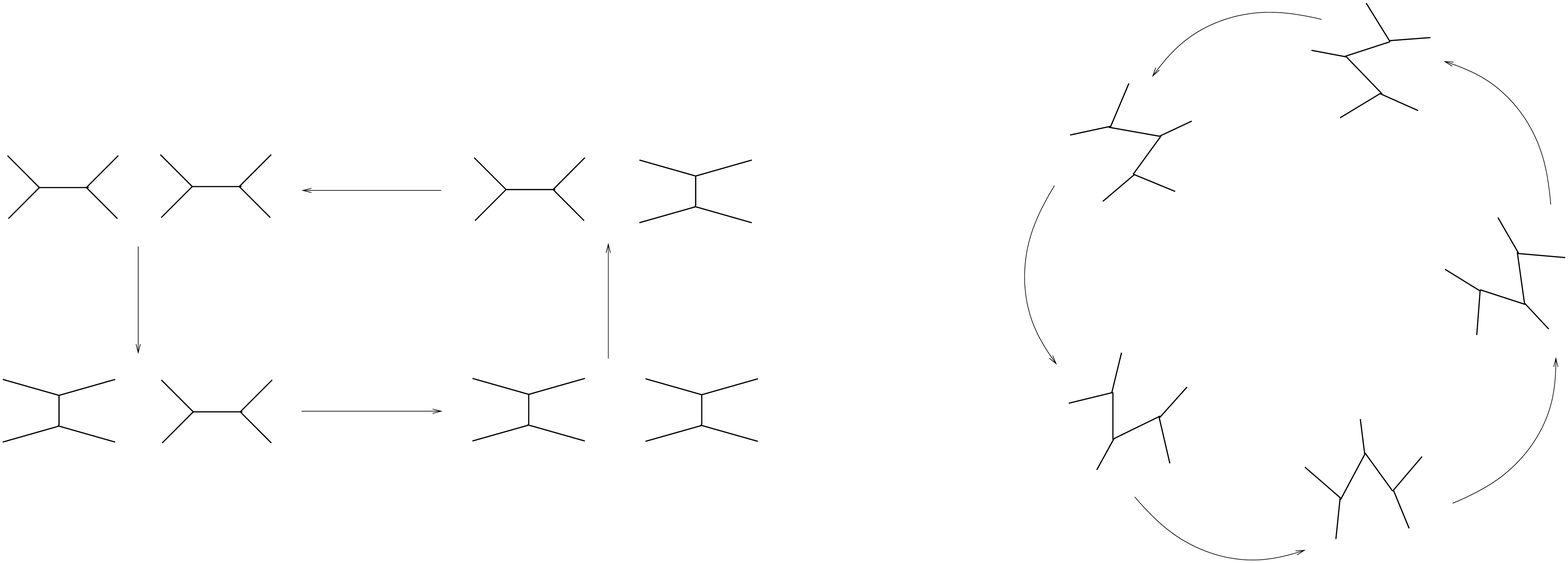}}
\caption{The commutativity and the pentagon relations.}
\label{fig:relations}
\end{figure}
In addition to the relations defined by the $2$-cells,
there are the involutivity relations which, in the groupoid $\Pt$, we write
$$
\labellist \small \hair 2pt
\endlabellist
\centering
\includegraphics[scale=0.3]{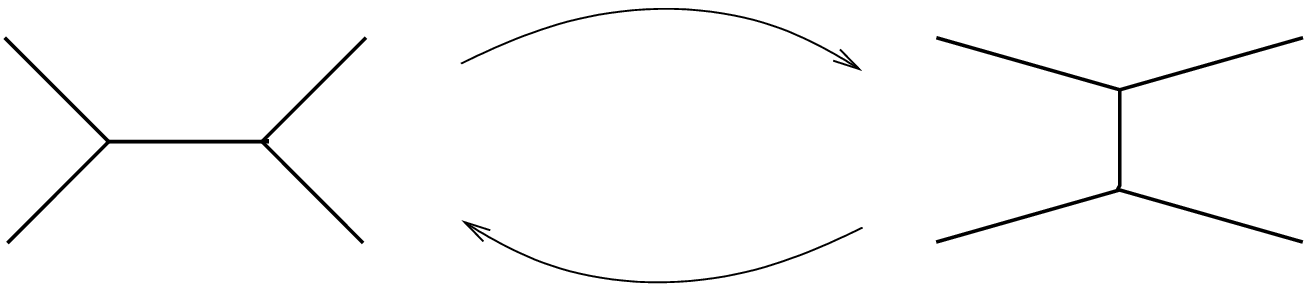}.
$$
Thus, the Ptolemy groupoid consists of finite sequences of Whitehead moves between 
$\pi$-marked trivalent bordered fatgraphs, modulo the involutivity, commutativity and pentagon relations.

\begin{remark}
\label{rem:mcg_to_Pt}
Since the space $\MFat$ is arc-connected, any two $\pi$-marked trivalent bordered fatgraphs  are related
by a finite sequence of Whitehead moves and, since $\MFat$ is contractible, 
this sequence is unique up  to the involutivity, commutativity and pentagon relations. 
In particular, any $\pi$-marked trivalent bordered fatgraph $G$ induces an injective map
$$
\xymatrix{
G: \Mcg\ \ar@{>->}[r] & \Pt
}
$$
which sends any $f \in \Mcg$ to a finite sequence of Whitehead moves relating $G$ to $f(G)$.
\end{remark}

In general, we can associate to any normal subgroup $N$ of $\Mcg$ the cellular fundamental groupoid of the CW-complex $\MFat/N$ 
discussed at the end of \S \ref{subsec:fatgraphs}.
In particular, the \emph{$k$-th Torelli groupoid} of $\Sigma$ is defined, for every $k\geq 1$, as
$$
\Pt/\Mcg[k] := \pi_1^{\operatorname{cell}}\left(\MFat/\Mcg[k]\right)
$$
where $\Mcg[k]$ is the $k$-th term of the Johnson filtration.
The cases $k=0$ and $k=1$ are of special interest. 
We get the \emph{mapping class groupoid} $\Pt/\Mcg$ and the \emph{Torelli groupoid} $\Pt/\Torelli$.
Observe that, in sharp contrast with $\Fat=\MFat/\Mcg$, the category $\MFat/\Torelli$ is not finite. 
To describe its objects, we set 
$$
H:= H_1(\Sigma;\Z).
$$
An \emph{$H$-marking} of a bordered fatgraph $G$ is a map $\eta:\{\hbox{oriented edges of $G$}\} \to H$ 
which is induced by a $\pi$-marking of $G$. This implies the following:
\begin{itemize}
\item For every oriented edge $e$ of $G$, we have $\eta(e) + \eta(\overline{e})=0$;
\item For every internal vertex $v$ of $G$, we have $\eta(h_1^+) + \cdots + \eta(h_k^+)=0$
where $h_1,\dots,h_k$ are the half-edges incident to $v$
and $h_1^+,\dots,h_k^+$ denote the corresponding edges which are oriented towards $v$;
\item The image of $\eta$ generates $H$.
\end{itemize}
(Those three conditions do not suffice to characterize $H$-markings: 
a simple criterion to recognize $H$-markings of $G$ among maps $\{\hbox{oriented edges of $G$}\} \to H$ is given in \cite{BKP}.) 
With that terminology, the category $\MFat/\Torelli$ has $H$-marked bordered fatgraphs for objects.
Thus, $\Pt/\Torelli$ consists of finite sequences of Whitehead moves between $H$-marked trivalent bordered fatgraphs, 
modulo the involutivity, commutativity and pentagon relations.

\section{Tautological extensions of Morita homomorphisms}

\label{sec:tautological}

In this section, we define for every $k\geq 1$ an extension $\widetilde{M}_k$
of the $k$-th Morita homomorphism $M_k$ to the $k$-th Torelli groupoid.
When the Ptolemy groupoid is interpreted in terms of triangulations of the cylinder $\Sigma \times [-1,1]$,
the homomorphism $\widetilde{M}_k$ appears as a natural extension of $M_k$ whose original definition  is simplicial by nature \cite{Morita_abelian}.
Thus, we call  $\widetilde{M}_k$ the ``tautological'' extension of $M_k$.
Our constructions apply to the Ptolemy groupoid, 
so that they also produce extensions of Morita's homomorphisms to the \emph{full} mapping class group.

\subsection{Construction}

We start by recalling Morita's definition of his homomorphisms \cite{Morita_abelian}.
For this, we consider the bar complex of the group $\pi$:
\begin{equation}
\label{eq:bar}
\xymatrix{
\cdots \ar[r] & \Bar_3(\pi) \ar[r]^-{\partial_3} & \Bar_2(\pi) \ar[r]^-{\partial_2} 
& \Bar_1(\pi) \ar[r]^-{\partial_1} & \Bar_0(\pi) \ar[r] & 0
}
\end{equation}
where $\Bar_n(\pi) := \Z \cdot \pi^{\times n}$ is the free abelian group generated by $n$-uplets 
$\left(\gp{g}_1| \cdots | \gp{g}_n  \right)$ of elements of $\pi$ and
$$
\partial_n\left(\gp{g}_1| \cdots | \gp{g}_n  \right) := 
(\gp{g}_2 | \cdots | \gp{g}_n)
+ \sum_{i=1}^{n-1} (-1)^i \cdot (\gp{g}_1 | \cdots |\gp{g}_i \gp{g}_{i+1}| \cdots |\gp{g}_n)
+ (-1)^n \cdot (\gp{g}_1 |\cdots| \gp{g}_{n-1}). 
$$
Let $\zeta:=[\partial \Sigma] \in \pi$ be the homotopy class of the boundary curve.
Since $\zeta$ is trivial in $\pi/[\pi,\pi] \simeq H_1(\pi)$, 
we can find a $2$-chain $Z \in \Bar_2(\pi)$ such that $\partial_2(Z)=-(\zeta)$.
For any $f\in \Mcg(\Sigma)$, 
the $2$-chain $Z-f_*(Z)$ is a cycle since $f_*(\zeta)=\zeta$. Then, knowing that $H_2(\pi)=0$,
we can find a $3$-chain $T_f \in \Bar_3(\pi)$ such that $\partial_3(T_f)=Z-f_*(Z)$.
If we now assume that $f \in \Mcg[k]$, the reduction of $T_f$ to $\pi/\Gamma_{k+1} \pi$ is a cycle,
and we can set
$$
M_k(f) := \left[T_f \mod \Gamma_{k+1} \pi\right] \ \in H_3\left(\pi/\Gamma_{k+1} \pi\right).
$$
It can be checked that this homology class does not depend 
on the choice of $T_f$ nor on the choice of $Z$, 
and that $M_k(f_1 \circ f_2)=M_k(f_1) + M_k(f_2)$ for all $f_1,f_2 \in \Mcg[k]$: see \cite{Morita_abelian}.

\begin{definition}
For all $k\geq 1$, the map $M_k:\Mcg[k] \to H_3\left(\pi/\Gamma_{k+1} \pi\right)$
is called the \emph{$k$-th Morita homomorphism.}
\end{definition}

Here is a first observation to extend Morita's homomorphisms to the Ptolemy groupoid.

\begin{lemma}
\label{lem:Z_G}
Any $\pi$-marked trivalent bordered fatgraph $G$ defines a $2$-chain
$Z_G \in \Bar_2(\pi)$ with boundary $-(\zeta) \in \Bar_1(\pi)$.
\end{lemma}

\begin{proof}
There is a preferred orientation on the edges of $G$.
Indeed, if $\mathbb{G}$ is the oriented surface defined by $G$ (as recalled in \S \ref{sec:review})
and if one follows the oriented boundary of $\mathbb{G}$, starting from the external vertex,
then each edge $e$ of $G$ is approached two times
and the first passage gives the preferred orientation of $e$:

\begin{equation}
\label{eq:edge_orientation}
\labellist \small \hair 2pt
\pinlabel {first passage} [t] at 350 0
\pinlabel {second passage} [b] at 350 236
\pinlabel {$e$} [b] at 451 122
\pinlabel {$\partial \mathbb{G}$} [r] at 1 3
\pinlabel {$\partial \mathbb{G}$} [r] at 1 219
\endlabellist
\centering
\includegraphics[scale=0.18]{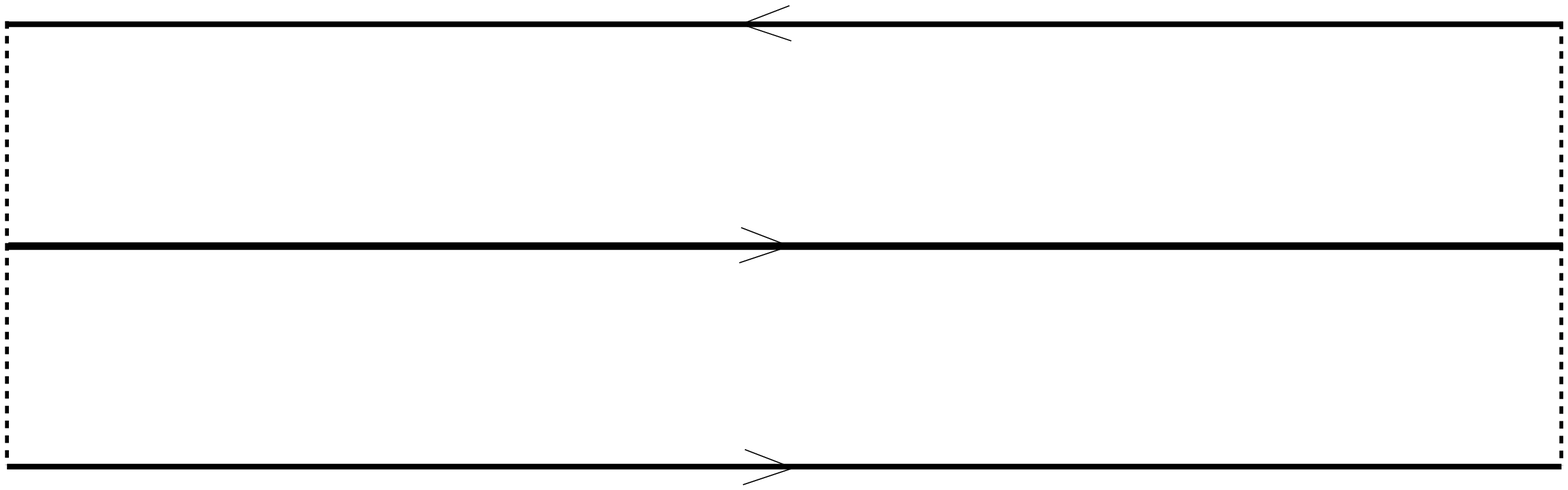}
\end{equation}\\[-0.4cm]

\noindent
The $\pi$-marking $\varpi$ of $G$ defines a $2$-chain in the bar complex of $\pi$ by the sum
$$
Z_G := \sum_{i} \varepsilon(i) \cdot \big(\gp{l}(i)\big\vert \gp{r}(i)\big) \ \in\Bar_2(\pi)
$$
which is taken over all internal vertices $i$ of $G$.
Here, $\varepsilon(i) \in \{+1,-1\}$, $\gp{l}(i) \in \pi$ and $\gp{r}(i) \in \pi$ 
depend on the orientations of the half-edges and the values of $\varpi$ around $i$:
\begin{equation}
\label{eq:Z_G}
{\labellist \small \hair 2pt
\pinlabel {$\gp{l}(i)$} [rt] at 16 85
\pinlabel {$\gp{r}(i)$} [lt] at 75 85
\pinlabel {$\gp{l}(i)$} [rb] at 491 16
\pinlabel {$\gp{r}(i)$} [lb] at 550 16
\pinlabel {$\varepsilon(i):=+1$} [t] at 45 -4
\pinlabel {$\varepsilon(i):=-1$} [t] at 523 -4
\pinlabel {or} at 305 47
\endlabellist
\includegraphics[scale=0.4]{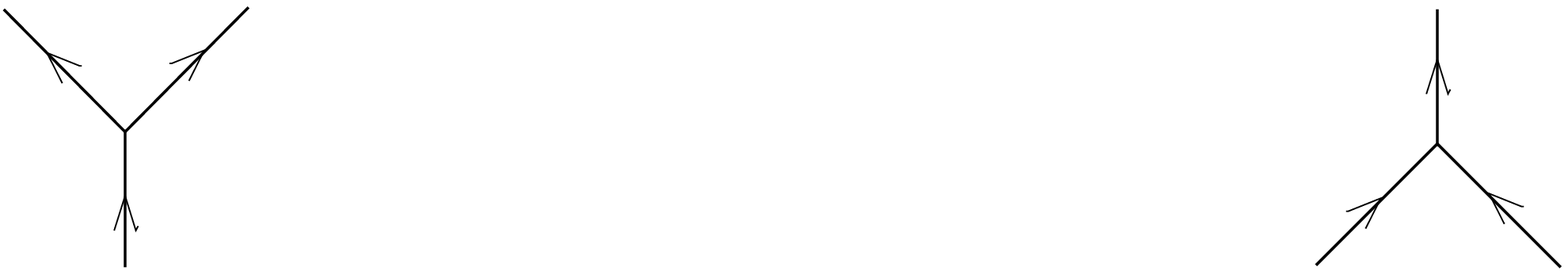}}
\end{equation}\\[-0.4cm]

\noindent
Thus, the contribution of an internal vertex $i$ to $\partial_2(Z_G)$
is $\gp{r}(i) -\gp{l}(i)\gp{r}(i) +\gp{l}(i)$ if $\varepsilon(i)=+1$, 
and it is $-\gp{r}(i) +\gp{l}(i)\gp{r}(i) -\gp{l}(i)$ if $\varepsilon(i)=-1$.
Therefore, we have
$$
\partial_2(Z_G) = \sum_{h} \epsilon(h) \cdot \varpi(h^+)
$$
where the sum is taken over all half-edges $h$ of $G$ that are incident to an internal vertex,
$\epsilon(h) \in \{+1,-1\}$ is positive if and only if $h$ is outgoing,
and $\varpi(h^+) \in \pi$ is the value assigned by $\varpi$ to the edge $h^+$ containing $h$. 
All  terms of $\partial_2(Z_G)$ vanish by pairs except for the half-edge close to the external vertex, 
which is incoming and  colored by $\zeta\in \pi$.
\end{proof}

Here is our second  observation to extend Morita's homomorphisms to the Ptolemy groupoid.

\begin{lemma}
\label{lem:M}
Any Whitehead move $W:G \to G'$ between $\pi$-marked trivalent bordered fatgraphs  
defines a $3$-chain $T_W \in \Bar_3(\pi)$ with boundary $Z_G-Z_{G'} \in \Bar_2(\pi)$. 
\end{lemma}

\begin{proof}
Consider the $\pi$-marked bordered fatgraph $Q$ with a single $4$-valent vertex $q$ that corresponds to the Whitehead move $W$.
When one follows $\partial \mathbb{Q}$ starting from the external vertex of $Q$,
the order $0<1<2<3$ into which the $4$ sectors of $q$ are approached can be (up to a cyclic permutation) of $(4!)/4=6$ types.
Thus, the Whitehead move $W$ can be of $6\cdot 2=12$ types. We set 
$$
T_W := s \cdot \left(\gp{k}|\gp{h}|\gp{g}\right) \ \in \Bar_3(\pi)
$$
where the values of $s \in \{+1,-1\}$ and $\gp{g},\gp{h},\gp{k} \in \pi$ are shown on Figure \ref{fig:T_W} for each  type. 
In each case, one can check by a simple computation that $\partial_3(T_W)=Z_G-Z_{G'}$.
\end{proof}

\begin{figure}[h]
{\labellist \small \hair 0pt
\pinlabel {$q=$ \ } [r] at 0 551
\pinlabel {$q=$ \ } [r] at 0 447
\pinlabel {$q=$ \ } [r] at 0 344
\pinlabel {$q=$ \ } [r] at 0 240
\pinlabel {$q=$ \ } [r] at 0 135
\pinlabel {$q=$ \ } [r] at 0 31
\pinlabel {$1$}  at 10 551
\pinlabel {$1$}  at 10 447
\pinlabel {$2$}  at 10 344
\pinlabel {$2$}  at 10 240
\pinlabel {$3$}  at 10 135
\pinlabel {$3$}  at 10 31
\pinlabel {$2$}  at 53 576
\pinlabel {$3$}  at 53 472
\pinlabel {$1$}  at 53 369
\pinlabel {$3$}  at 53 265
\pinlabel {$1$}  at 53 160
\pinlabel {$2$}  at 53 56
\pinlabel {$0$}  at 101 551
\pinlabel {$0$}  at 101 447
\pinlabel {$0$}  at 101 344
\pinlabel {$0$}  at 101 240
\pinlabel {$0$}  at 101 135
\pinlabel {$0$}  at 101 31
\pinlabel {$3$}  at 53 526
\pinlabel {$2$}  at 53 422
\pinlabel {$3$}  at 53 319
\pinlabel {$1$}  at 53 215
\pinlabel {$2$}  at 53 110
\pinlabel {$1$}  at 53 6
\pinlabel {$\displaystyle{\mathop{\longrightarrow}^{W}_{s := +1}}$}  at 380 576
\pinlabel {$\displaystyle{\mathop{\longrightarrow}^{W}_{s := -1}}$}  at 380 472
\pinlabel {$\displaystyle{\mathop{\longrightarrow}^{W}_{s := -1}}$}  at 380 369
\pinlabel {$\displaystyle{\mathop{\longrightarrow}^{W}_{s := +1}}$}  at 380 265
\pinlabel {$\displaystyle{\mathop{\longrightarrow}^{W}_{s := +1}}$}  at 380 160
\pinlabel {$\displaystyle{\mathop{\longrightarrow}^{W}_{s := -1}}$}  at 380 56
\pinlabel {$\displaystyle{\mathop{\longleftarrow}_{W}^{s := -1}}$}  at 380 526
\pinlabel {$\displaystyle{\mathop{\longleftarrow}_{W}^{s := +1}}$}  at 380 422
\pinlabel {$\displaystyle{\mathop{\longleftarrow}_{W}^{s := +1}}$}  at 380 319
\pinlabel {$\displaystyle{\mathop{\longleftarrow}_{W}^{s := -1}}$}  at 380 215
\pinlabel {$\displaystyle{\mathop{\longleftarrow}_{W}^{s := -1}}$}  at 380 110
\pinlabel {$\displaystyle{\mathop{\longleftarrow}_{W}^{s := +1}}$}  at 380 6
\pinlabel {$\gp{g}$} [r] at 480 551
\pinlabel {$\gp{h}$} [bl] at 244 572
\pinlabel {$\gp{k}$} [b] at 275 558 
\pinlabel {$\gp{g}$} [r] at 478 446
\pinlabel {$\gp{h}$} [tl] at 244 429
\pinlabel {$\gp{k}$} [b] at 277 454
\pinlabel {$\gp{g}$} [br] at 311 361
\pinlabel {$\gp{h}$} [bl] at 237 367
\pinlabel {$\gp{k}$} [br] at 233 325
\pinlabel {$\gp{g}$} [tr] at 312 219
\pinlabel {$\gp{h}$} [tl] at 240 219
\pinlabel {$\gp{k}$} [bl] at 241 263
\pinlabel {$\gp{g}$} [tl] at 325 155
\pinlabel {$\gp{h}$} [b] at 276 141
\pinlabel {$\gp{k}$} [tl] at 243 118
\pinlabel {$\gp{g}$} [tr] at 309 13
\pinlabel {$\gp{h}$} [b] at 277 36
\pinlabel {$\gp{k}$} [bl] at 243 50
\endlabellist
\includegraphics[scale=0.6]{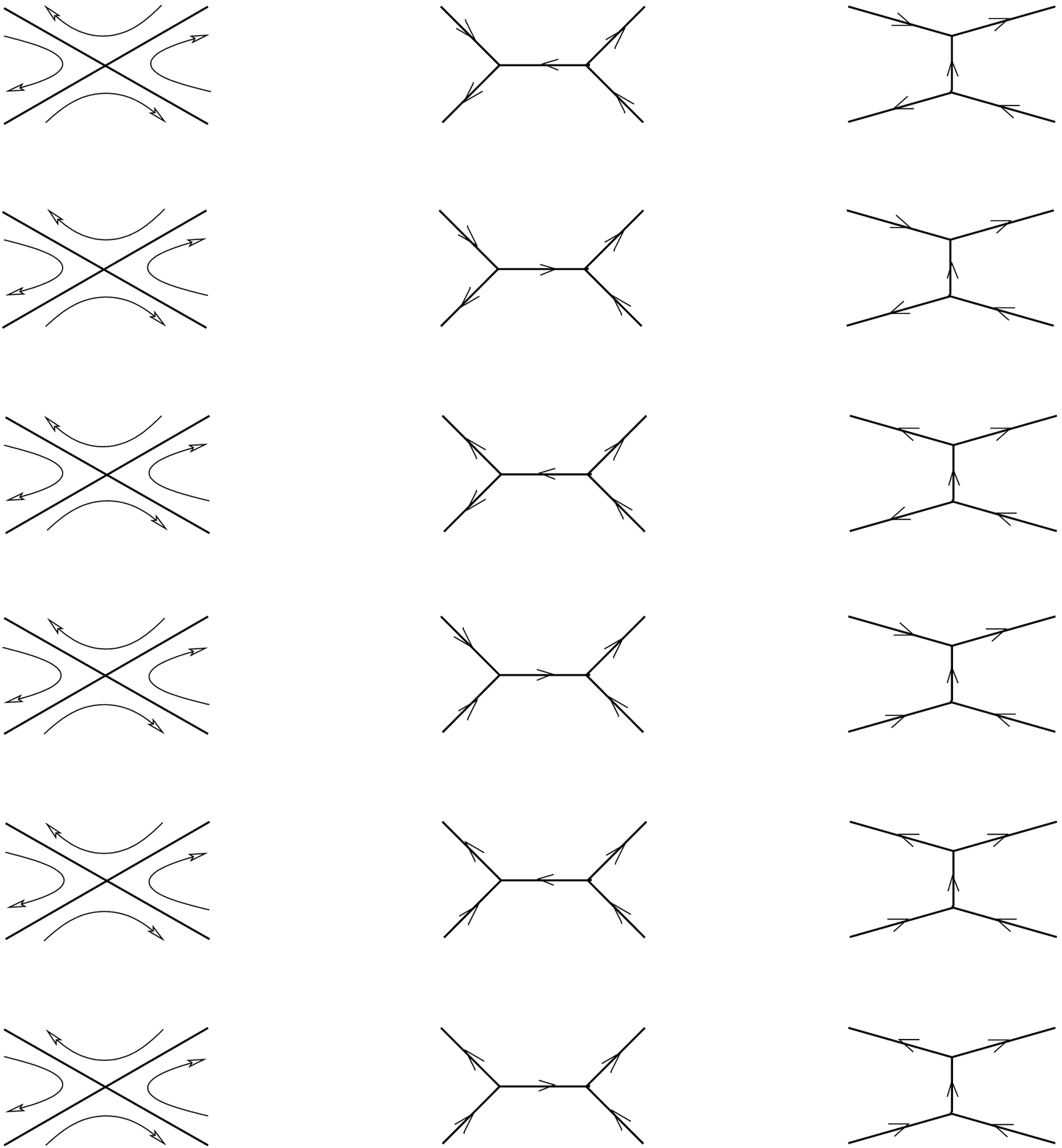}}
\caption{Values for the definition of the $3$-chain $T_W$.}
\label{fig:T_W}
\end{figure}

We can now define our ``tautological'' extensions of Morita's homomorphisms to  $\Pt$.

\begin{theorem}
\label{th:tautological}
There exists a unique groupoid homomorphism
$$
\widetilde{M}: \Pt \longrightarrow \Bar_3(\pi)/\Img(\partial_4)
$$
defined by $W \mapsto T_W$  on each Whitehead move $W$, and $\widetilde{M}$ is $\Mcg$-equivariant.
Moreover, the groupoid homomorphism $\widetilde{M}_k$, defined for all $k\geq 1$ by the commutative diagram
$$
\xymatrix{
\Pt \ar@{->>}[d]  \ar[r]^-{\widetilde{M}}  & \frac{\Bar_3(\pi)}{\Img(\partial_4)} \ar@{->>}[d]\\
\Pt/\Mcg[k] \ar@{-->}[r]_-{\widetilde{M}_k} & \frac{\Bar_3(\pi/\Gamma_{k+1} \pi)}{\Img(\partial_4)},
}
$$
is a groupoid extension of $M_k:\Mcg[k] \to H_3\left(\pi/\Gamma_{k+1} \pi\right)$.
\end{theorem}

\begin{proof}
For any sequence of Whitehead moves $S=(G_1 \overset{W_1}{\to} G_2 \overset{W_2}{\to} \cdots \overset{W_r}{\to} G_{r+1})$
among $\pi$-marked trivalent bordered fatgraphs, 
the $3$-chain $T_{W_1} + T_{W_2} + \cdots + T_{W_r}$ has boundary $Z_{G_1} - Z_{G_{r+1}}$.
Since $H_3(\pi)$ is trivial, it follows that the class of this $3$-chain modulo $\Img(\partial_4)$
is  determined by the initial and final points of $S$. 
Therefore, the pentagon and commutativity relations are automatically satisfied,
so that the groupoid map $\widetilde{M}$ is well-defined.
The $\Mcg$-equivariance of $\widetilde{M}$ is easily checked.

Thus, the map $\widetilde{M}$ induces a groupoid map 
$\widetilde{M}_k:\Pt/\Mcg[k] \to \Bar_3(\pi/\Gamma_{k+1} \pi)/\Img(\partial_4)$ for every $k \geq 1$.
To check that $\widetilde{M}_k$ is an extension of  $M_k$,
let $f\in \Mcg[k]$ and represent $f$ as a sequence of Whitehead moves 
$(G=G_1 \overset{W_1}{\to} G_2 \overset{W_2}{\to} \cdots \overset{W_r}{\to} G_{r+1}=f(G))$. We have
$$
\partial_3\left(T_{W_1} + T_{W_2} + \cdots + T_{W_r}\right)
= Z_{G_1} - Z_{G_{r+1}} = Z_{G} - f_*\left( Z_{G} \right)
$$
so, by definition of $M_k$, we have 
\begin{eqnarray*}
\hspace{2cm} M_k(f) &=& \left[T_{W_1} + T_{W_2} + \cdots + T_{W_r} \mod \Gamma_{k+1} \pi\right]\\
&=& \widetilde{M}_k\left((G_1 \overset{W_1}{\to} G_2 \overset{W_2}{\to} \cdots \overset{W_r}{\to} G_{r+1}) \mod \Mcg[k]\right).  
\end{eqnarray*}
Therefore, the following square is commutative:
\begin{equation}
\label{eq:tautological_extension}
\newdir{ >}{{}*!/-5pt/\dir{>}}
\xymatrix{
\pi_1\left(\MFat/\Mcg[k],\{G\}\right) \ar@{^{(}->}[d] \ar@{=}[r]^-G & \Mcg[k] \ar[r]^-{M_k} & H_3\left(\pi/\Gamma_{k+1}\pi\right) \ar@{^{(}->}[d] \\
\pi_1^{\operatorname{cell}}\left(\MFat/\Mcg[k]\right) \ar@{=}[r]  & \Pt/\Mcg[k] \ar[r]_-{\widetilde{M}_k}
& \frac{\Bar_3(\pi/\Gamma_{k+1} \pi)}{\Img(\partial_4)}
} 
\end{equation}
We conclude that $\widetilde{M}_k$ is a groupoid extension of $M_k$. More precisely,
the groupoid extension problem as it is formulated in the Introduction, with 
$$
\Gamma :=\Mcg[k], \ K:= \MFat/\Mcg[k]
\quad \hbox{and} \quad
A:=H_3(\pi/\Gamma_{k+1}\pi), \  \varphi := M_k, 
$$
has for solution
$$
\widetilde{A} := \Bar_3(\pi/\Gamma_{k+1} \pi)/\Img(\partial_4), \quad \widetilde{\varphi} := \widetilde{M}_k.
$$
In this situation, the $\Gamma$-action on $A$ and the $\pi_1^{\operatorname{cell}}(K)$-action on $\widetilde{A}$ are trivial.
\end{proof}

\begin{remark}
\label{rem:infinite}
It should be noticed that the target of $\widetilde{M}_k$ is a free abelian group of infinite-rank. 
Indeed, we have the following exact sequence:
$$
\xymatrix{
0 \ar[r] & H_3\left(\pi/\Gamma_{k+1} \pi\right) \ar[r] & 
\frac{\Bar_3(\pi/\Gamma_{k+1} \pi)}{\Img(\partial_4)}  \ar[r]^-{\partial_3} & \Ker(\partial_2)  \ar[r] &  H_2\left(\pi/\Gamma_{k+1} \pi\right)\ar[r] &0.
}
$$
The group  $H_2\left(\pi/\Gamma_{k+1} \pi\right)$  is finitely generated free abelian
(as follows from Hopf's theorem) and the same is true for $H_3\left(\pi/\Gamma_{k+1} \pi\right)$ (as proved in \cite{IO}).
The abelian group $\Ker(\partial_2)$ is free (as a subgroup of $\Bar_2(\pi/\Gamma_{k+1} \pi)$) and is not finitely generated
(since, for any $\gp{x}\neq 1$, the $2$-cycles $c_{i} := (\gp{x}^i|\gp{x}^{i+1})-(\gp{x}^{i+1}|\gp{x}^i)$ are linearly independent).
It follows that the abelian group $\Bar_3(\pi/\Gamma_{k+1} \pi)/\Img(\partial_4)$ is free and has infinite rank.
\end{remark}

The previous construction can be applied to extend Morita's homomorphisms to the mapping class group.
For this, we shall \emph{choose} a $\pi$-marked trivalent bordered fatgraph $G$.
We  consider the composition
$$
\xymatrix{
\Mcg\ \ar@{>->}[r]^-G  \ar@/_1.3pc/@{-->}[rr]_-{\widetilde{M}_G}
& \Pt \ar[r]^-{\widetilde{M}} & \Bar_3(\pi)/\Img(\partial_4)   
}
$$
where the first map is defined in Remark \ref{rem:mcg_to_Pt}.

\begin{corollary}
\label{cor:mapping_class_group}
For every $k\geq 1$, the $k$-th reduction $\widetilde{M}_{G,k}$ of $\widetilde{M}_G$ defined by 
$$
\xymatrix{
\Mcg \ar[r]^-{\widetilde{M}_G} \ar@/_1.5pc/@{-->}[rr]_-{\widetilde{M}_{G,k}} & \Bar_3(\pi)/\Img(\partial_4)  \ar@{->>}[r] 
& \Bar_3(\pi/\Gamma_{k+1} \pi)/\Img(\partial_4)
}
$$
is an extension of the $k$-th Morita homomorphism to the mapping class group.
Moreover, $\widetilde{M}_{G,k}$ is a crossed homomorphism whose homology class
$$
\left[\widetilde{M}_{G,k}\right] \in H^1\left(\Mcg;\frac{\Bar_3(\pi/\Gamma_{k+1} \pi)}{\Img(\partial_4)}\right)
$$
does not depend on the choice of $G$.
\end{corollary}

\begin{proof}
The first part follows from Theorem \ref{th:tautological} and  the commutativity of the diagram
$$
\xymatrix{
\Mcg\ \ar@{>->}[rr]^-G &  & \pi_1^{\operatorname{cell}}\left(\MFat\right) = \Pt   \ar@{->>}[d] \\
\Mcg[k] \ar@{=}[r]^-G \ar@{^{(}->}[u]  & \pi_1\left(\MFat/\Mcg[k],\{G\}\right) \ar@{^{(}->}[r]  
& \pi_1^{\operatorname{cell}}\left(\MFat/\Mcg[k]\right)  = \Pt/\Mcg[k].
}
$$
To prove the second statement, it is enough to check that the map $\widetilde{M}_G$ itself is a crossed homomorphism 
whose homology class does not depend on $G$. 
Let $f$ and $f' \in \Mcg$ and represent them by some sequences of Whitehead moves 
$S: G \to \cdots \to f(G)$ and $S':G \to \cdots \to f'(G)$. 
So $f'(S)$ is a sequence of Whitehead moves $f'(G) \to \cdots \to f' \circ f(G)$.
The $\Mcg$-equivariance of $\widetilde{M}$ gives
$$
\widetilde{M}_G(f'\circ f) = \widetilde{M}\left(S' \hbox{ followed by } f'(S)\right)
= \widetilde{M}(S') + f' \cdot \widetilde{M}(S) =\widetilde{M}_G(f')+ f' \cdot \widetilde{M}_G(f),
$$
\ie $\widetilde{M}_G$ is a crossed homomorphism.
Assume now that $G'$ is another $\pi$-marked trivalent bordered fatgraph, 
and let $U$ be a sequence of Whitehead moves connecting $G'$ to $G$. 
For all $f\in \Mcg$, the $\Mcg$-equivariance of $\widetilde{M}$ implies  
$$
\widetilde{M}_{G'}(f)= \widetilde{M}\left( U \hbox{ followed by } S \hbox{ followed by } f(U^{-1})\right)
= \widetilde{M}(U) + \widetilde{M}_G(f) - f\cdot \widetilde{M}(U).
$$
As a result, the $1$-cocycles $\widetilde{M}_G$ and $\widetilde{M}_{G'}$ differ by a coboundary.
\end{proof}

\subsection{Topological interpretation}

We shall now give a topological description of the groupoid map $\widetilde{M}$ defined in Theorem \ref{th:tautological}.
This is based on the $3$-dimensional interpretation of the Ptolemy groupoid in terms of triangulations and Pachner moves.
The triangulations  that we shall consider will be  \emph{singular} in the sense of \cite{TV}, and we will use the following terminology.  
An \emph{edge orientation} of a triangulation  $\mathcal{T}$ is the choice of an orientation for each $1$-simplex of $\mathcal{T}$;
we call it \emph{admissible} if none of the $2$-simplices of $\mathcal{T}$ has an orientation compatible with all its edges.

First of all, we give a topological description of the $2$-chain $Z_G$ defined in Lemma \ref{lem:Z_G} 
for every $\pi$-marked trivalent bordered fatgraph $G$. Such a fatgraph defines a triangulated surface $\Sigma_G$:
one associates to each internal vertex $v$ one copy $\Delta^v$ of the standard $2$-dimensional simplex, 
and one glues $\Delta^v$ linearly to $\Delta^{v'}$ along one edge if $v$ and $v'$ are adjacent in $G$.
The triangulation of $\Sigma_G$ is singular, with a single vertex denoted by $\star$.
There is a canonical embedding of the ``thickened'' graph $\mathbb{G}$ into $\Sigma_G$,
such that $G\subset \mathbb{G}$ is dual to the triangulation of $\Sigma_G$ and 
$\mathbb{G} \cap \partial \Sigma_G$ is the ``free'' end $F$ of $\mathbb{G}$. In particular, the surface $\Sigma_G$ has a preferred orientation.
As for the marking of $G$, it defines an isotopy class of orientation-preserving homeomorphisms $\Sigma_G \to \Sigma$ sending $F$ to $I$. 
In the sequel, that identification $\Sigma_G\cong \Sigma$ defined by the marking of $G$ will be implicit.  
Note that the corresponding $\pi$-marking $\varpi$ of $G$ then satisfies $\varpi(e):= [e^*] \in \pi_1(\Sigma,\star)$ for all oriented internal edge $e$ of $G$,
where $e^*$ is the $1$-simplex of $\Sigma_G$ dual to $e$ and is oriented in such a way that $e^*\cdot e = +1$:
\begin{equation}
\label{eq:loop_orientation}
{\labellist \small \hair 2pt
\pinlabel {$\star$}  at 1 76
\pinlabel {$\star$}  at 217 76
\pinlabel {$\star$}  at 109 148
\pinlabel {$\star$}  at 108 4
\pinlabel {$e$} [b] at 85 76
\pinlabel {$e^*$} [r] at 110 117
\endlabellist
\includegraphics[scale=0.5]{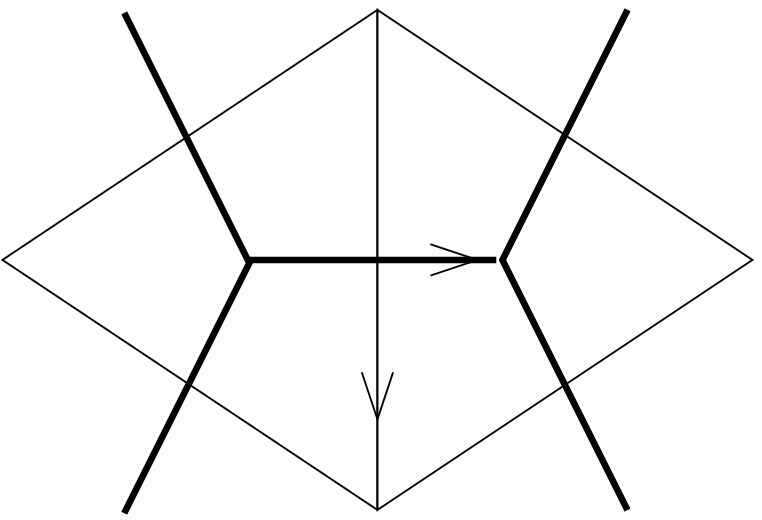}}
\end{equation}

Since each edge $e$ of $G$ has a preferred orientation (\ref{eq:edge_orientation}), 
each $1$-simplex $e^*$ of $\Sigma_G$ has a preferred orientation as shown on (\ref{eq:loop_orientation}). 
Thus $\Sigma_G$ has a preferred edge orientation, which is easily seen to be admissible.
We also get a preferred orientation on each $2$-simplex $\Delta^v$ of $\Sigma_G$, that is
\begin{equation}
\label{eq:triangle_orientation}
{\labellist \small \hair 2pt
\pinlabel {$\star$}  at 9 0
\pinlabel {$\star$}  at 9 182
\pinlabel {$\star$}  at 166 94
\pinlabel {$\star$}  at 384 183
\pinlabel {$\star$\ .} [l] at 532 94
\pinlabel {$\star$}  at 384 1
\pinlabel {or} at 280 91
\pinlabel {\Large $\circlearrowleft$} at 53 96
\pinlabel {\scriptsize $+$} at 53 96
\pinlabel {\Large $\circlearrowright$} at 440 96
\pinlabel {\scriptsize $+$} at 440 96
\endlabellist
\includegraphics[scale=0.3]{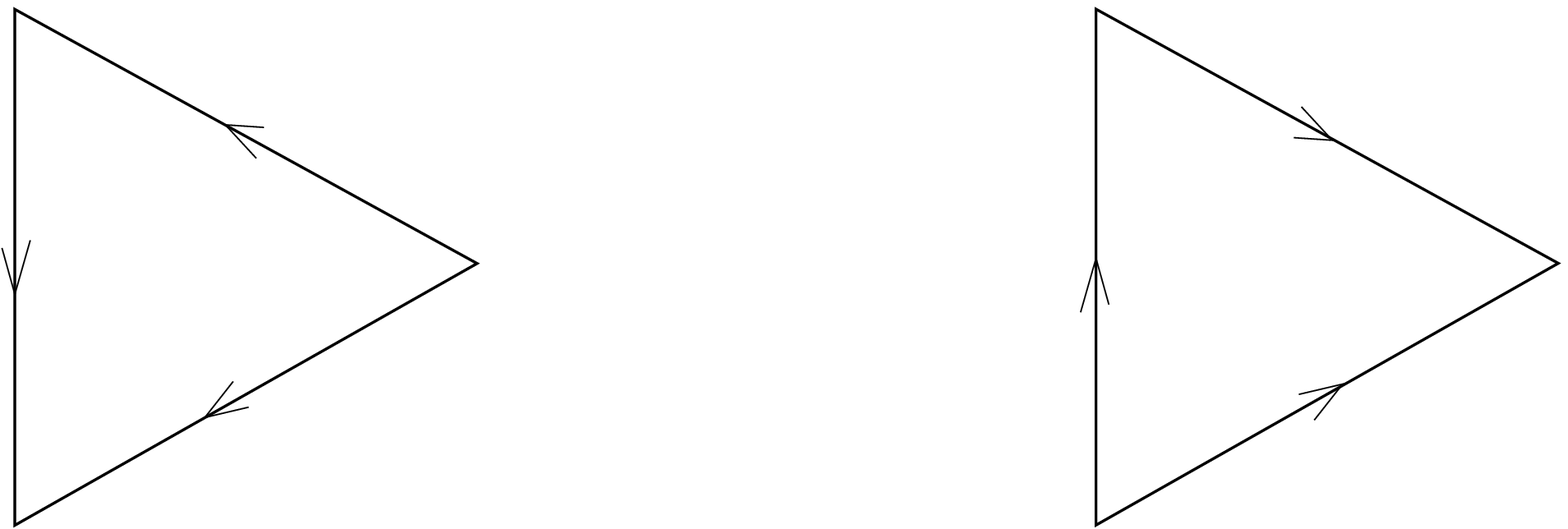}}
\end{equation}
Thus, we can consider the cell chain complex of $\Sigma_G$ and the fundamental class $[\Sigma_G]$
defines a $2$-chain in that complex.

Let $\K^\Delta(\pi,1)$ be the simplicial model for the Eilenberg--MacLane space
(whose construction is recalled below), which we equip  with its canonical singular triangulation.
Recall that the cell chain complex of $\K^\Delta(\pi,1)$ is the bar complex (\ref{eq:bar}) of $\pi$.

\begin{lemma}
\label{lem:description_Z_G}
For any $\pi$-marked trivalent bordered fatgraph $G$, there is a canonical cellular map
$h_G: \Sigma_{G} \to \K^\Delta(\pi,1)$ inducing the canonical isomorphism at the level of fundamental groups.
Moreover, we have $h_G([ \Sigma_{G}]) = - Z_G \ \in \Bar_2(\pi).$
\end{lemma}

\noindent
It follows that $\partial_2(Z_G) = -h_G\left(\partial_2([ \Sigma_{G}])\right)= -h_G([\partial  \Sigma_{G}])= -(\zeta)$, as expected.

\begin{proof}
We recall how $\K^\Delta(\pi,1)$ is defined
by first constructing its universal cover $\widetilde{\K^\Delta}(\pi,1)$, see \cite{Brown_groups} for instance.
To each $\sigma=(\gp{g}_0,\dots,\gp{g}_n) \in \pi^{n+1}$, we associate a copy $\Delta^\sigma$
of the standard $n$-dimensional simplex: its vertices are denoted $(0^\sigma,\dots,n^\sigma)$.
Set $d_i\sigma := (\gp{g}_0,\dots,\widehat{\gp{g}_i},\dots,\gp{g}_n)$ $\in \pi^{n}$ 
and let $\delta_i^\sigma: \Delta^{d_i\sigma} \to \Delta^\sigma$ 
be the affine embedding sending $0^{d_i\sigma},\dots,{(n-1)}^{d_i\sigma}$ to 
the vertices $0^\sigma, \dots, \widehat{i^{\sigma}},\dots,n^{\sigma}$ respectively. We define
$$
\widetilde{\K^\Delta}(\pi,1) :=  \bigsqcup_{\substack{n\geq 0,\ \sigma\in \pi^{n+1}}}\! \! \!  \Delta^{\sigma} \bigg/ \sim
$$
where the equivalence relation $\sim$ on the disjoint union identifies each $\Delta^{d_i\sigma}$ with the $i$-th face of $\Delta^\sigma$ via $\delta_i^\sigma$.
The group $\pi$ acts on the left of $\widetilde{\K^\Delta}(\pi,1)$: for all $\gp{g}\in \pi$ and $\sigma=(\gp{g}_0,\dots,\gp{g}_n) \in \pi^{n+1}$,
the simplex $\Delta^{\sigma}$ is mapped linearly to $\Delta^{\gp{g}\cdot \sigma}$ 
where $\gp{g} \cdot \sigma := (\gp{gg}_0,\dots,\gp{gg}_n)$, 
the vertices $0^\sigma,\dots,n^\sigma$ being mapped to $0^{\gp{g}\cdot \sigma},\dots,n^{\gp{g}\cdot \sigma}$ respectively.
This action is free and $\widetilde{\K^\Delta}(\pi,1)$ is contractible so that
$$
\K^\Delta(\pi,1) := \widetilde{\K^\Delta}(\pi,1)/\pi
$$
is an  Eilenberg--MacLane space of type $\K(\pi,1)$ and $\widetilde{\K^\Delta}(\pi,1)$ is its universal cover.
The (singular) triangulation of $\widetilde{\K^\Delta}(\pi,1)$ projects onto a (singular) triangulation of $\K^\Delta(\pi,1)$,
whose associated cell chain complex  is the bar complex $\Bar_*(\pi)$.
More precisely, the oriented $n$-dimensional cell $(\gp{g}_1|\gp{g}_2|\cdots|\gp{g}_n)$ is the projection
of the $n$-dimensional simplex $\Delta^{(1,\gp{g}_1,\gp{g}_1\gp{g}_2,\dots,\gp{g}_1 \gp{g}_2 \dots \gp{g}_n)}$ 
of $ \widetilde{\K^\Delta}(\pi,1)$ with its canonical orientation.

We now define a cellular map $h_G: \Sigma_G \to \K^\Delta(\pi,1)$.
For each internal vertex $v$ of $G$, we consider the $2$-simplex $\Delta^v$ of $\Sigma_G$ \emph{before gluing}
and, for the sake of clarity, we denote it by $\Delta^{v}_{\operatorname{b.g}}$: 
thus, the projection $\Delta^{v}_{\operatorname{b.g}} \to \Delta^{v}$ is a desingularization of $\Delta^{v}$.
The preferred edge orientation of $\Sigma_G$ lifts to $\Delta^v_{\operatorname{b.g}}$.
Thus, the $3$ vertices of $\Delta^v_{\operatorname{b.g}}$ have a total ordering 
which allows us to denote them by $0^v,1^v,2^v$:

\begin{equation}
\label{eq:total_ordering}
{\labellist \small \hair 2pt
\pinlabel {$2^v$}  [tr] at 5 0
\pinlabel {$1^v$} [br] at 5 182
\pinlabel {$0^v$}  [l] at 166 94
\pinlabel {$1^v$}  [br] at 384 183
\pinlabel {$2^v$}  [l] at 546 94
\pinlabel {$0^v$} [tr] at 384 1
\pinlabel {or} at 280 91
\endlabellist
\includegraphics[scale=0.3]{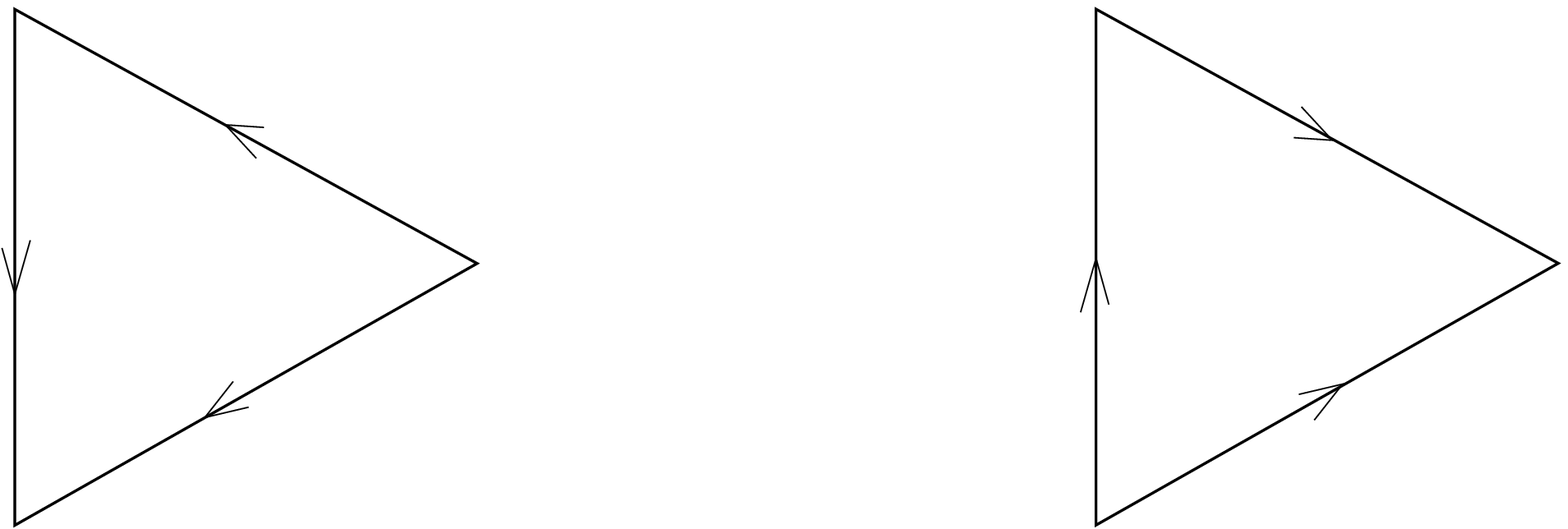}}
\end{equation}
Let $\gp{g}_{01}^v\in \pi$ and $\gp{g}_{12}^v \in \pi$ be the homotopy classes of the looped edges of $\Delta^v$ (after gluing) that correspond
to the $1$-simplices $(0^v,1^v)$ and $(1^v,2^v)$ of $\Delta^v_{\operatorname{b.g}}$, respectively.
We also consider the $2$-simplex $\Delta^{\sigma(v)}$ of $\widetilde{\K^\Delta}(\pi,1)$ \emph{before gluing}, where we set
$$
\sigma(v) := (1,\gp{g}_{01}^v,\gp{g}_{01}^v\gp{g}_{12}^v) \in \pi^3.
$$ 
Again, we denote it by $\Delta^{\sigma(v)}_{\operatorname{b.g}}$
and the projection $\widetilde{\K^\Delta}(\pi,1) \to \K^\Delta(\pi,1)$ 
induces a desingularization $\Delta^{\sigma(v)}_{\operatorname{b.g}} \to (\gp{g}_{01}^v|\gp{g}_{12}^v)$.
The affine isomorphism $\Delta^v_{\operatorname{b.g}} \to \Delta^{\sigma(v)}_{\operatorname{b.g}}$ defined by $i^v \mapsto i^{\sigma(v)}$ 
induces a map $\Delta^v \to (\gp{g}_{01}^v|\gp{g}_{12}^v)$.
Doing this for each internal vertex $v$ of $G$, we obtain a cellular map $h_G: \Sigma_G \to \K^\Delta(\pi,1)$.
Note that each $1$-simplex $e^*$ of $\Sigma_G$ is sent by  $h_G$ to the $1$-simplex $([e^*])$ of $\K^\Delta(\pi,1)$:
so, $h_G$ induces the desired isomorphism at the level of fundamental groups.

We now  compute the image by $h_G$ of the fundamental class of $\Sigma_G$, which is the sum
$$
[\Sigma_G] = \sum_v \varepsilon'(v) \cdot \Delta^v \ \in C_2(\Sigma_G)
$$
over all internal vertices $v$ of $G$. Here the sign $\varepsilon'(v)=\pm 1$ 
compares the preferred orientation (\ref{eq:triangle_orientation}) of $\Delta^v$ with the orientation of the surface $\Sigma_G$.
Using the previous notations, we have 
$$
h_G\left(\left[\Sigma_{G}\right]\right) = \sum_v \varepsilon'(v) \cdot (\gp{g}_{01}^v \vert \gp{g}_{12}^v) \ \in \Bar_2(\pi).
$$
We see from (\ref{eq:Z_G}) that $\varepsilon(v)= -\varepsilon'(v)$, $\gp{g}_{01}^v=\gp{l}(v)$ and $\gp{g}_{12}^v=\gp{r}(v)$.
We conclude that $h_G([\Sigma_G]) = - Z_G$.
\end{proof}

\begin{remark}
Lemma \ref{lem:description_Z_G} is a variation of a well-known fact:
given an $n$-dimensional closed oriented manifold $M$,
there is a recipe to find representatives in $\Bar_n(\pi_1(M))$ of the image of $[M]$ in $H_n(\pi_1(M))$.
The main ingredient is a (possibly singular) triangulation of $M$ with an admissible edge orientation.
For instance, this is used in the construction of Dijkgraaf--Witten invariants  \cite{DW,AC}.
\end{remark}

The topological interpretation of the groupoid map $\widetilde{M}$ needs the cylinder $\Sigma \times [-1,1]$,
which we think of in its ``lens'' form:

$$
L := (\Sigma\times [-1,1])/\!\sim\ = \quad
\begin{array}{c}
{\labellist \small \hair 2pt
\pinlabel {$ \partial \Sigma $} [r] at 2 54
\pinlabel {$\partial \Sigma $} [l] at 265 54
\pinlabel {$\Sigma_+ $} [b] at 134 104
\pinlabel {$\Sigma_- $} [t] at 134 0
\endlabellist
\includegraphics[scale=0.3]{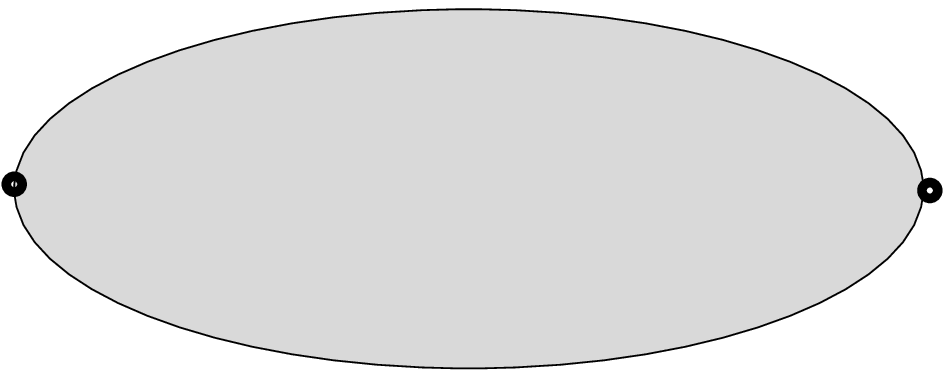}}
\end{array}
$$
\vspace{0.1cm}

\noindent
where the equivalence relation $\sim$ collapses $\partial \Sigma \times [-1,1]$ to $\partial \Sigma \times \{0\}$,
and $\Sigma_{\pm}$ denotes the image of $\Sigma \times \{\pm 1\}$ under that identification.

\begin{lemma}
\label{lem:L_G}
Any $\pi$-marked trivalent bordered fatgraph $G$ defines a singular triangulation $L_G$ of $L$
with a preferred admissible edge orientation, which restricts to $\Sigma_G$ on $\Sigma_+$ and $\Sigma_-$. 
(The triangulation $L_G$ has a single vertex $\star$ and $12g-1$ tetrahedra.) 
\end{lemma}

\begin{proof}
The ``lens'' cylinder $L$ can be identified to
\begin{equation}
\label{eq:union}
\frac{\Sigma \times [-1,+1]}{\partial \Sigma \times \{-1\} = \partial \Sigma \times\{+1\}}
\cup \left(\partial \Sigma \times D^2\right)
\end{equation}
where $\partial D^2$ is identified with the quotient space $[-1,+1]/(\{-1\} =\{+1\})$.
The  torus $\partial \Sigma \times \partial D^2$ can be given the following triangulation:

\begin{equation}
\label{eq:torus}
\begin{array}{c}
{\labellist \small \hair 2pt
\pinlabel {$ \partial D^2 $} [r] at 1 108
\pinlabel {$\partial D^2$} [l] at 299 98
\pinlabel {$\partial \Sigma$} [b] at 145 190
\pinlabel {$\partial \Sigma$} [t] at 145 0
\pinlabel {$\star$} at 4 6
\pinlabel {$\star$} at 292 6
\pinlabel {$\star$} at 292 185 
\pinlabel {$\star$} at 4 185
\endlabellist
\includegraphics[scale=0.3]{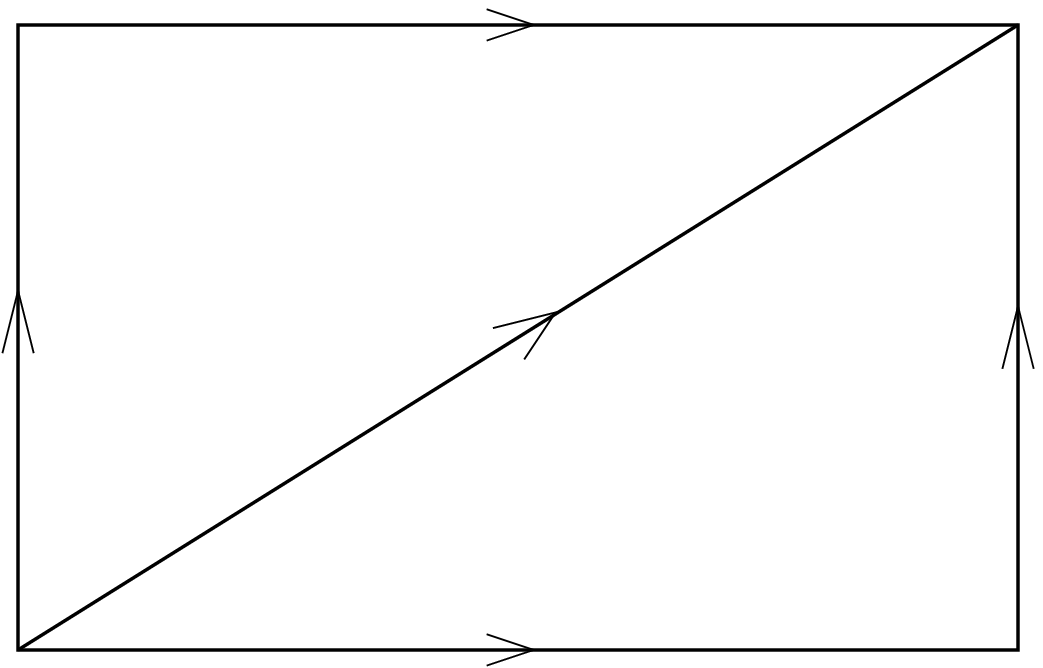}}
\end{array}
\end{equation}
The most economic way to extend it to a one-vertex triangulation of the solid torus $\partial \Sigma \times D^2$ is by adding $2$ tetrahedra, 
as explained in \cite[Figure 5.C]{JR}. Furthermore, it can be checked that the edge orientation shown on (\ref{eq:torus}) can be extended in a unique way
to an admissible edge orientation on that triangulation of $\partial \Sigma \times D^2$.

Next, we  triangulate the cylinder $\Sigma \times [-1,1]$ by subdividing, 
for each internal vertex $v$ of $G$, the prism $\Delta^v \times [-1,1]$ into $3$ tetrahedra.
The way that we subdivide is determined by the preferred edge  orientation of $\Sigma_G$ -- 
see Figure \ref{fig:prism}. The ``non-horizontal'' edges of the prisms are oriented upwards. 
Thus, $G$ defines a two-vertex triangulation of $\Sigma \times [-1,1]$, with a preferred admissible edge orientation
and $3\cdot(4g-1)=12g-3$ tetrahedra. 
Gluing it with the above edge-oriented triangulation of $\partial \Sigma \times D^2$, 
we obtain a one-vertex  edge-oriented triangulation $L_G$ of $L$ with $12g-1$ tetrahedra.
\end{proof}

\begin{figure}
{\labellist \small \hair 2pt
\pinlabel {or} at 188 109
\pinlabel {$\star$} at  60 5
\pinlabel {$\star$} at  1 63
\pinlabel {$\star$} at  120 64
\pinlabel {$\star$} at  61 122
\pinlabel {$\star$} at  2 180
\pinlabel {$\star$} at  119 180
\pinlabel {$\star$} at  316 3
\pinlabel {$\star$} at  256 61
\pinlabel {$\star$} at  375 61
\pinlabel {$\star$} at  315 120
\pinlabel {$\star$} at  257 178
\pinlabel {$\star$} at  374 179
\endlabellist
\includegraphics[scale=0.4]{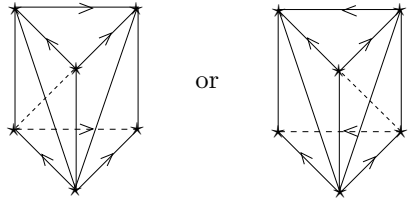}}
\caption{How to triangulate the prism $\Delta_v \times [-1,1]$.}
\label{fig:prism}
\end{figure}

A Whitehead move $G \stackrel{W}{\to} G'$ corresponds, at the level of singular triangulations,
to a diagonal exchange, \ie to a $2 \leftrightarrow 2$ Pachner move $\Sigma_G \stackrel{W}{\to} \Sigma_{G'}$.
Thus, a Whitehead move can be seen as a tetrahedron $\Delta^W$ making the transition 
between the triangulated surfaces $\Sigma_G$ and $\Sigma_{G'}$, as shown in Figure \ref{fig:Whitehead}.
Therefore, a sequence of Whitehead moves between $\pi$-marked trivalent bordered fatgraphs 
$$
S=\left(G_-=G_1 \stackrel{W_1}{\longrightarrow} G_2 \stackrel{W_2}{\longrightarrow} 
\cdots \stackrel{W_r}{\longrightarrow}  G_{r+1}=G_+\right)
$$ 
defines a singular triangulation $L_S$ of the ``lens'' cylinder $L$ by gluing successively $r$ tetrahedra $\Delta^{W_1},\dots,\Delta^{W_r}$
to the triangulation $L_{G_-}$ (defined in Lemma \ref{lem:L_G}).
This triangulation $L_S$ has only one vertex $\star$,
and restricts to $\Sigma_{G_\pm}$ on the surface $\Sigma_{\pm}$.
It has a preferred edge orientation, since $L_{G_-}$ and each triangulated surface $\Sigma_{G_i}$ have one. 
That edge orientation of $L_S$ is admissible and it induces an orientation for each triangle of $L_S$ by the rule (\ref{eq:triangle_orientation}).
Furthermore, the edge orientation of $L_S$ also induces  an orientation for each tetrahedron of $L_S$ using the following rule:
\begin{equation}
\label{eq:tetrahedron_orientation}
{\labellist \small \hair 2pt
\pinlabel {$\star$} at 4 136
\pinlabel {$\star$} at 153 136
\pinlabel {$\star$} at 167 196
\pinlabel {$\star$} at 80 238
\pinlabel {$\star$} at 359 134
\pinlabel {$\star$} at 508 134
\pinlabel {$\star$} at 523 191
\pinlabel {$\star$} at 434 237
\pinlabel {\scriptsize $+$} at 122 29
\pinlabel {\Large $\circlearrowleft$} at 122 29
\pinlabel {\scriptsize $+$} at 434 44
\pinlabel {\Large $\circlearrowleft$} at 434 44
\pinlabel {or} at 269 109
\pinlabel {``upwards''} [r] at 72 101
\pinlabel {``upwards''} [l] at 518 70
\endlabellist
\includegraphics[scale=0.5]{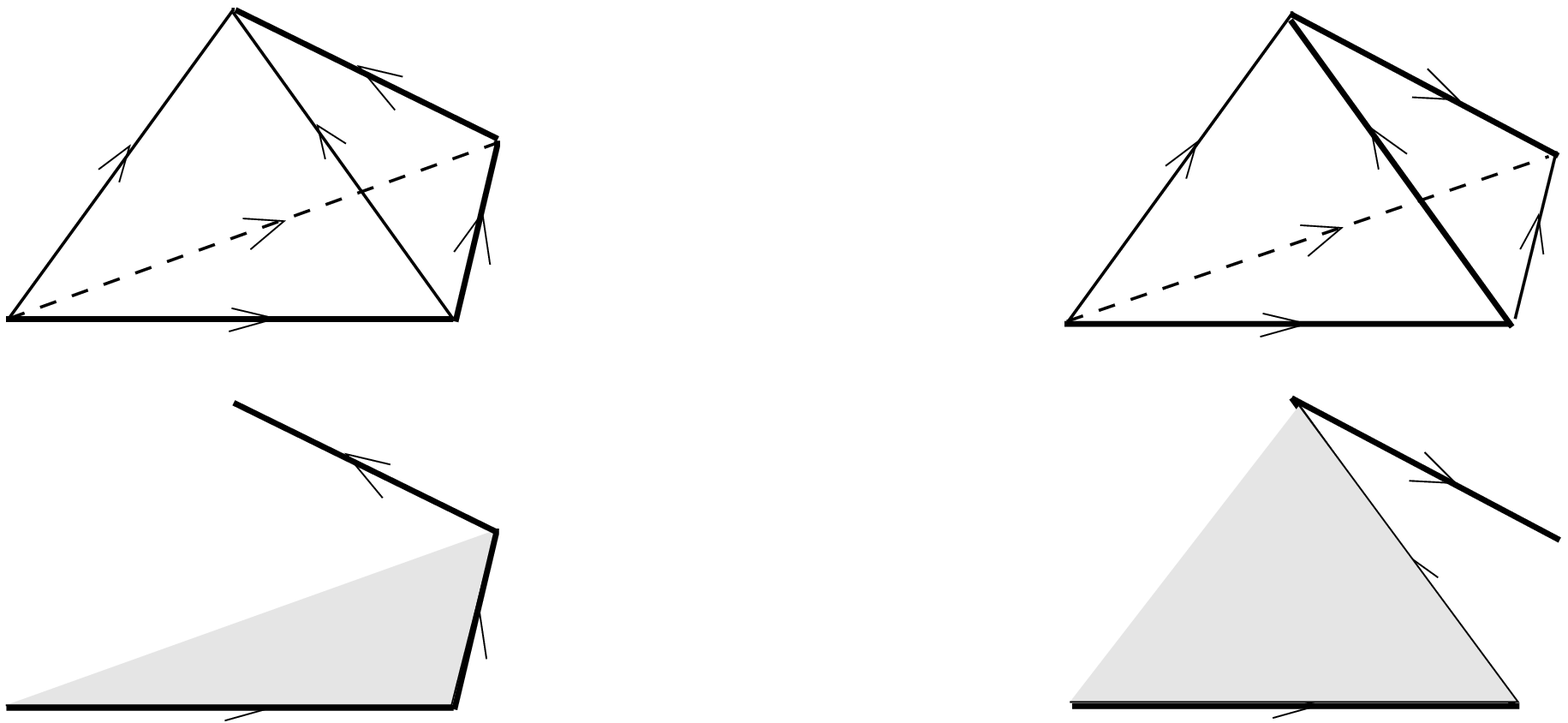}}
\end{equation}
(Here, an orientation of a tetrahedron is defined by orienting one of its faces 
and by specifying a normal vector to that face, which is said to point ``upwards''.) 
Thus, all simplices of $L_S$ have been given an orientation so that
we can consider the cell chain complex of $L_S$.
The fundamental class $[L_S]$ defines a $3$-chain in that complex.

\begin{figure}
{\labellist \small \hair 2pt
\pinlabel {$\stackrel{W}{\longrightarrow}$}  at 253 217
\pinlabel {$\Sigma_G$} [bl] at 116 266
\pinlabel {$\Sigma_{G'}$} [rb] at 385 266
\pinlabel {\scriptsize \begin{tabular}{r}  ``invisible''\\ face\end{tabular}} [tl] at 50 98
\pinlabel {\scriptsize \begin{tabular}{l}  ``visible''\\ face\end{tabular}} [tr] at 450 98
\pinlabel {$\Delta^W$} [bl] at 268 140
\pinlabel {$\star$} at 1 218 
\pinlabel {$\star$} at 145 218 
\pinlabel {$\star$} at 73 291
\pinlabel {$\star$} at 73 147
\pinlabel {$\star$} at 433 290
\pinlabel {$\star$} at 361 219
\pinlabel {$\star$} at 433 146
\pinlabel {$\star$} at 506 218
\endlabellist
\includegraphics[scale=0.4]{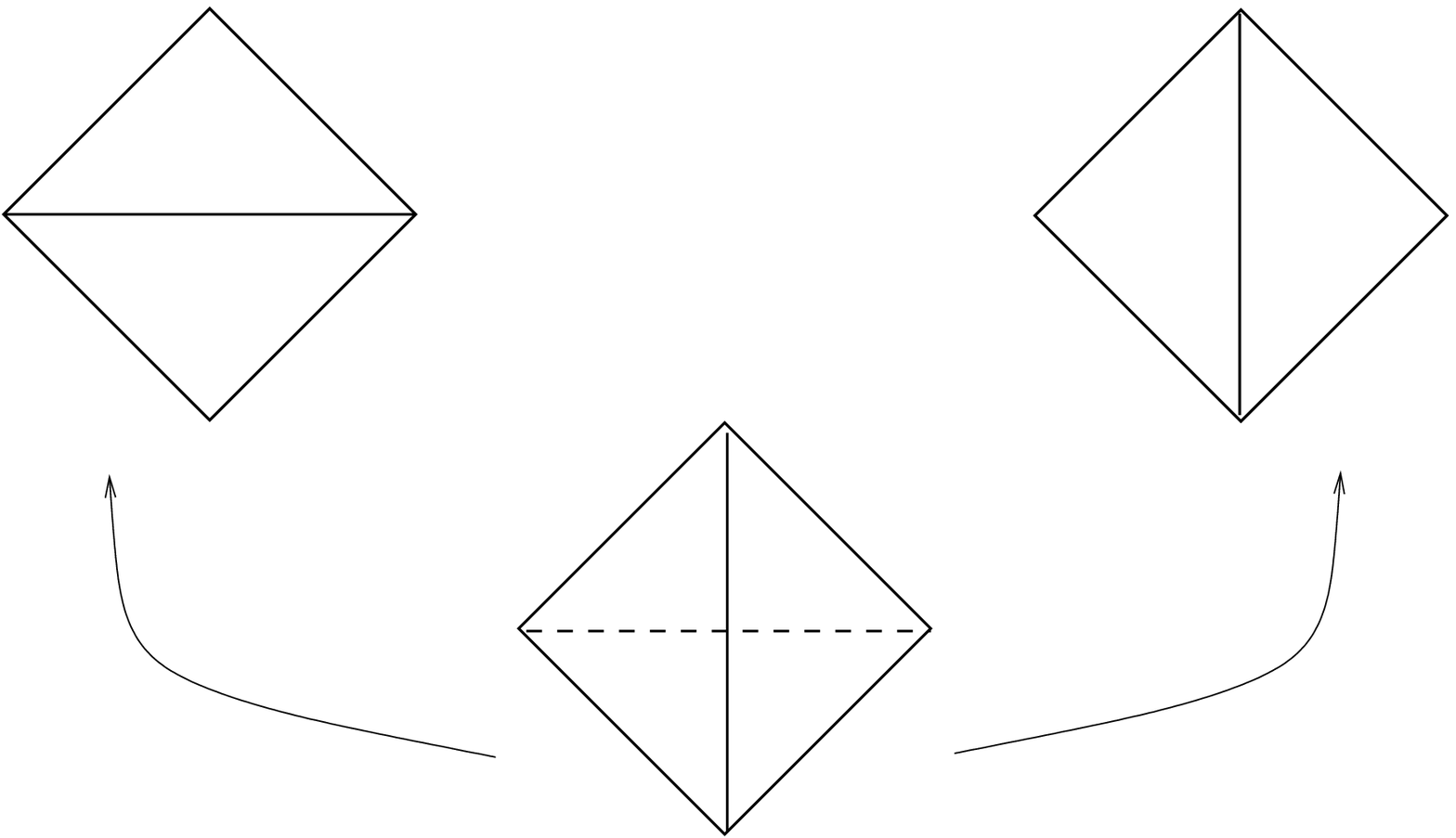}}
\caption{The $3$-dimensional interpretation of a Whitehead move.}
\label{fig:Whitehead}
\end{figure}

\begin{proposition}
\label{prop:T_S}
For any sequence of Whitehead moves among $\pi$-marked trivalent bordered fatgraphs
$$
S=\left(G_-=G_1 \stackrel{W_1}{\longrightarrow} G_2 \stackrel{W_2}{\longrightarrow} 
\cdots \stackrel{W_r}{\longrightarrow}  G_{r+1}=G_+\right),
$$ 
there is a canonical cellular map $h_S: L_S \to \K^\Delta(\pi,1)$
inducing the canonical isomorphism at the level of fundamental groups. Moreover, we have
$$
h_S\left([L_S]\right) = \left(T_{W_1} + \cdots + T_{W_r}\right)
+ b_{G_-} \ \in \Bar_3(\pi)
$$
where $b_{G_-}$ only depends on $G_-$ and belongs to $\Img(\partial_4)$.
\end{proposition}

\noindent
It follows that $\widetilde{M}(\{S\}) \in \Bar_3(\pi)/\Img(\partial_4)$ is represented by $h_S\left([L_S]\right)$, 
and this is our topological interpretation of the map $\widetilde{M}$.

\begin{proof}
The cellular map $h_S: L_S \to \K^\Delta(\pi,1)$ is defined in a way similar to the map $h_G$ in Lemma \ref{lem:Z_G}.
Thus, we consider each tetrahedron $\Delta^t$ of $L_S$ before gluing, 
and the corresponding  desingularization map $\Delta_{\operatorname{b.g}}^t \to \Delta^t$.
The admissible edge orientation of $L_S$ induces a total ordering on the set of vertices of $\Delta_{\operatorname{b.g}}^t$,
which allows us to denote them by $0^t,1^t,2^t,3^t$:

$$
{\labellist \small \hair 2pt
\pinlabel {or} at 273 59
\pinlabel {$0^t$} [tr] at 4 8
\pinlabel {$1^t$} [tl] at 152 8
\pinlabel {$2^t$} [bl] at 168 70
\pinlabel {$3^t$} [b] at 80 110
\pinlabel {$0^t$} [tr] at 360 7
\pinlabel {$1^t$} [tl] at 509 5
\pinlabel {$2^t$} [br] at 436 108
\pinlabel {$3^t$} [tl] at 528 68
\endlabellist
\includegraphics[scale=0.5]{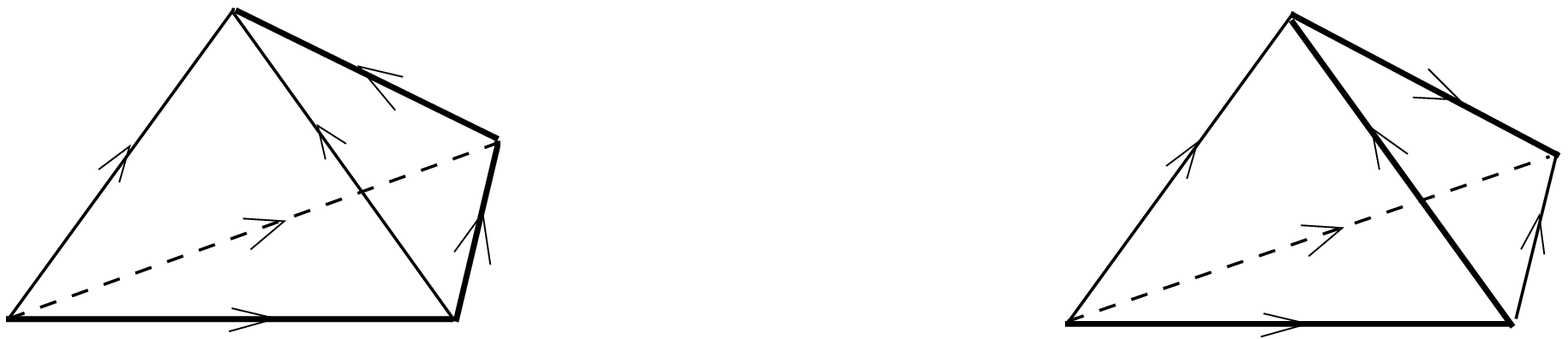}}
$$
\vspace{0.1cm}

\noindent
Let $\gp{g}_{01}^t,\gp{g}_{12}^t,\gp{g}_{23}^t\in \pi$ be the homotopy classes of the looped edges of $\Delta^t$
corresponding to the $1$-simplices $(0^t,1^t),(1^t,2^t),(2^t,3^t)$ of $\Delta_{\operatorname{b.g}}^t$.
We denote 
$$
\sigma(t) := (1,\gp{g}_{01}^t, \gp{g}_{01}^t\gp{g}_{12}^t,\gp{g}_{01}^t\gp{g}_{12}^t\gp{g}_{23}^t) \in \pi^4
$$ 
and we consider the desingularization map $\Delta^{\sigma(t)}_{\operatorname{b.g}}\to (\gp{g}_{01}^t|\gp{g}_{12}^t|\gp{g}_{23}^t)$
induced by the projection $\widetilde{\K^\Delta}(\pi,1) \to \K^\Delta(\pi,1)$.
The affine isomorphism $\Delta^t_{\operatorname{b.g}} \to \Delta^{\sigma(t)}_{\operatorname{b.g}}$ defined by $i^t \mapsto i^{\sigma(t)}$ 
induces a map $\Delta^t \to (\gp{g}_{01}^t|\gp{g}_{12}^t|\gp{g}_{23}^t)$.
Doing this for each tetrahedron $t$ of $L_S$, we obtain a cellular map $h_S: L_S \to \K^\Delta(\pi,1)$
which induces the canonical isomorphism  at the level of fundamental groups.

Using the same notations, we can compute the image by $h_S$ of the fundamental class 
$$
[L_S] = \sum_t \varepsilon'(t) \cdot \Delta^t \ \in C_3(L_S),
$$
where the sign $\varepsilon'(t)=\pm 1$ compares the orientation (\ref{eq:tetrahedron_orientation}) 
with the product orientation of the $3$-manifold $L_S \cong \Sigma \times [-1,1]$. We obtain
\begin{equation}
\label{eq:fundamental_class}
h_S\left([L_S]\right) = \sum_t \varepsilon'(t) \cdot (\gp{g}_{01}^t|\gp{g}_{12}^t|\gp{g}_{23}^t) \ \in \Bar_3(\pi). 
\end{equation}
Recall that the tetrahedra $\Delta^t$ of $L_S$ are of two types: some come from $L_{G_-}$
while the other come from the Whitehead moves $W_1,\dots,W_r$. 
Thus, the sum (\ref{eq:fundamental_class}) decomposes into two summands.
The first summand only depends on $G_-$ and is denoted by $b_{G_-}$.
It has boundary $Z_{G_-}-Z_{G_-}$ and, so, $b_{G_-}$ must belong to $\Img(\partial_4)$ since $H_3(\pi)=0$.
The second summand is exactly $T_{W_1} + \cdots + T_{W_r}$, as it can be checked by comparison with Figure \ref{fig:T_W}.
\end{proof}

Proposition \ref{prop:T_S} also provides a topological interpretation 
to the crossed homomorphism $\widetilde{M}_G: \Mcg(\Sigma) \to \Bar_3(\pi)/\Img(\partial_4)$.
The \emph{closure} $C_f$ of an $f\in \Mcg(\Sigma)$ is the manifold
obtained from the ``lens'' cylinder $L$ by gluing $\Sigma_+$ to $\Sigma_-$ with $f$. 
Thus, $C_f$ comes with an open book decomposition whose binding is connected.
(It is well-known that any connected closed oriented $3$-manifold can be obtained in that way.)
Recall from \cite{Thomas,Swarup} that the oriented homotopy type 
of $C_f$ is determined by its fundamental group
$$
\pi_1\left(C_f\right) :=\pi_1\left(C_f,\star\right)
\simeq \pi\left/\left\langle f_*(\gp{g})\cdot \gp{g}^{-1}\vert \gp{g}\in \pi\right\rangle_{\operatorname{normal}}\right.
$$
together with the homology class
$$
\mu\left(C_f\right) := 
h_*\left([C_f]\right) \ \in H_3\left(\pi_1(C_f)\right),
$$
where $h:C_f \to \K\left(\pi_1(C_f),1\right)$ 
is a map inducing an isomorphism at the level of fundamental groups.

\begin{corollary}
\label{cor:oriented_homotopy_type}
Let $f\in \Mcg(\Sigma)$ be represented by a sequence of Whitehead moves
\begin{equation}
\label{eq:f_sequence}
G=G_1 \stackrel{W_1}{\longrightarrow} G_2 \stackrel{W_2}{\longrightarrow} \cdots \stackrel{W_r}{\longrightarrow}  G_{r+1}=f(G)
\end{equation}
among $\pi$-marked trivalent bordered fatgraphs.
Then $\mu(C_f)$ is the reduction of
$\widetilde{M}_G(f)= \left[T_{W_1}+ \cdots + T_{W_r}\right] \in B_3(\pi)/\Img(\partial_4)$
to $B_3\left(\pi_1\left(C_f\right)\right)/\Img(\partial_4)$ by the projection $\pi \to \pi_1\left(C_f\right)$.
\end{corollary}

\begin{proof}
Denote by $S$ the sequence of Whitehead moves (\ref{eq:f_sequence}).
The triangulation $L_S$ induces a one-vertex triangulation 
$(C_f)_S$ of $C_f$ by gluing
the faces of $\Sigma_{f(G)} \cong \Sigma_+$ to the faces of $\Sigma_G \cong \Sigma_-$.
Moreover, the map $h_S: L_S \to \K^\Delta(\pi,1)$ defined in Proposition \ref{prop:T_S}
induces a cellular map $h_S: (C_f)_S \to \K^\Delta\left(\pi_1(C_f),1\right)$.
We conclude that
$$
\mu\left(C_f\right)
= h_{S,*}\left(\left[(C_f)_S\right] \right)
= \left[\hbox{reduction of }b_{G} + T_{W_1} + \cdots + T_{W_r} \right]
$$
is represented by the reduction of $T_{W_1} + \cdots + T_{W_r}\in B_3(\pi)$ to $B_3\left(\pi_1\left(C_f\right)\right)$.
\end{proof}

\begin{remark}
Assume  that $f$ belongs to $\Mcg[k]$ so that the group $\pi_1(C_f)/ \Gamma_{k+1}\pi_1(C_f)$ 
is canonically isomorphic to $\pi/\Gamma_{k+1} \pi$.
It follows from Corollary \ref{cor:oriented_homotopy_type} that
the class $M_k(f) \in H_3(\pi/\Gamma_{k+1}\pi)$ is the reduction of $\mu(C_f)$ 
to $H_3\left(\pi_1(C_f)/\Gamma_{k+1}\pi_1(C_f)\right)$.
This topological interpretation of the $k$-th Morita homomorphism is due to Heap \cite{Heap}.
\end{remark}

\section{Infinitesimal extensions of Morita homomorphisms}

\label{sec:infinitesimal}

We have obtained in \S \ref{sec:tautological} an extension 
$$
\widetilde{M}_k: \Pt/\Mcg[k] \longrightarrow \Bar_3(\pi/\Gamma_{k+1} \pi)/\Img(\partial_4)
$$
of $M_k$ to the Ptolemy groupoid. 
However, as noticed in Remark \ref{rem:infinite}, 
the target of $\widetilde{M}_k$ is a free abelian group of \emph{infinite} rank.
That contrasts with the fact that $H_3(\pi/\Gamma_{k+1} \pi)$, the target  of $M_k$, is finitely generated.
This section is aimed at correcting that defect and, for this, we will replace groups by their Malcev Lie algebras. 
For the reader's convenience, a few facts about Malcev Lie algebras and their homology are recalled in the Appendix.

\subsection{Construction}

Let $k\geq 1$ be an integer and denote by $\Malcev(\pi/\Gamma_{k+1} \pi)$
the Malcev Lie algebra of the group $\pi/\Gamma_{k+1} \pi$.
We defined in \cite{Massuyeau} an ``infinitesimal'' version
$$
m_k: \Mcg[k] \longrightarrow H_3\left(\Malcev(\pi/\Gamma_{k+1} \pi); \Q\right)
$$
of the $k$-th Morita homomorphism. Its definition is similar to $M_k$, but with bar complexes of groups
replaced by Koszul complexes of Lie algebras. 
It is proved in \cite{Massuyeau} that the homomorphisms $M_k$ and $m_k$ are equivalent:
\begin{equation}
\label{eq:M_to_m}
\xymatrix{
\Mcg[k] \ar[r]^-{M_k} \ar[rrd]_-{m_k} & H_3\left(\pi/\Gamma_{k+1} \pi \right)\ \ar@{>->}[r] 
& H_3\left(\pi/\Gamma_{k+1} \pi ; \Q\right) \ar[d]^-{\P}_-\simeq \\
& & H_3\left(\Malcev(\pi/\Gamma_{k+1} \pi) ; \Q\right) 
}
\end{equation}
Here, $\P$ is a canonical isomorphism due to Pickel \cite{Pickel}.
Suslin and Wodzicki introduced in \cite{SW} a canonical chain map
$$
\SW: \Bar_*(\pi/\Gamma_{k+1} \pi) \longrightarrow \Lambda^* \Malcev(\pi/\Gamma_{k+1} \pi),
$$
from the bar complex of $\pi/\Gamma_{k+1} \pi$ to the Koszul complex of $\Malcev(\pi/\Gamma_{k+1}\pi)$, 
which induces $\P$ in homology. (See the Appendix for more details.)
We need this chain map  in degree $3$ to define the  abelian subgroup
\begin{equation}
\label{eq:target}
T_k(\pi) := \frac{\SW_3\big(\Bar_3(\pi/\Gamma_{k+1}\pi)\big)+\Img(\partial_4)}{\Img(\partial_4)}
\end{equation}
of $\Lambda^3 \Malcev(\pi/\Gamma_{k+1}\pi)/\Img(\partial_4)$.

\begin{theorem}
\label{th:infinitesimal}
The map $\widetilde{m}_k$ defined by
$$
\xymatrix{
\Pt/\Mcg[k] \ar[r]^-{\widetilde{M}_k} \ar@/_1.2pc/@{-->}[rr]_-{\widetilde{m}_k} & 
\frac{\Bar_3(\pi/\Gamma_{k+1} \pi)}{\Img(\partial_4)} \ar[r]^-{\SW_3}&  T_k(\pi)
}
$$
is a groupoid extension of $M_k: \Mcg[k] \to H_3\left(\pi/\Gamma_{k+1}\pi\right)$
whose target $T_k(\pi)$ is a finitely generated free abelian group.
\end{theorem}

\begin{proof}
Diagram (\ref{eq:tautological_extension}) and the definition of $\widetilde{m}_k$ show that the following square is commutative:
\begin{equation}
\label{eq:infinitesimal_extension}
\newdir{ >}{{}*!/-5pt/\dir{>}}
\xymatrix{
\pi_1\left(\MFat/\Mcg[k],\{G\}\right) \ar@{^{(}->}[d] \ar@{=}[r]^-G & \Mcg[k] \ar[r]^-{M_k} & H_3\left(\pi/\Gamma_{k+1}\pi\right) \ar@{ >->}[d]^-{\SW_3} \\
\pi_1^{\operatorname{cell}}\left(\MFat/\Mcg[k]\right) \ar@{=}[r]  & \Pt/\Mcg[k] \ar[r]_-{\widetilde{m}_k}
& T_k(\pi)
} 
\end{equation}
The right-hand map is the restriction of $\SW_3: \Bar_3(\pi/\Gamma_{k+1} \pi)/\Img(\partial_4) \to T_k(\pi)$ 
to the subgroup $H_3\left(\pi/\Gamma_{k+1}\pi\right)$,
and it is injective by the following commutative diagram:
$$
\newdir{ >}{{}*!/-5pt/\dir{>}}
\xymatrix{
H_3\left(\pi/\Gamma_{k+1}\pi\right)   \ar@{ >->}[d]   \ar@{^{(}->}[r] & \frac{\Bar_3(\pi/\Gamma_{k+1} \pi)}{\Img(\partial_4)} \ar[r]_{\SW_3} & 
T_k(\pi) \ar@{^{(}->}[d]  \\
H_3\left(\pi/\Gamma_{k+1}\pi;\Q\right) \ar[r]^-\simeq_-P & H_3\left(\Malcev(\pi/\Gamma_{k+1} \pi); \Q\right) \ar@{^{(}->}[r] 
& \frac{\Lambda^3 \Malcev(\pi/\Gamma_{k+1} \pi)}{ \Img(\partial_4)}
}
$$
Thus, (\ref{eq:infinitesimal_extension}) exactly tells us that $\widetilde{m}_k$ is an extension of $M_k$ to the Ptolemy groupoid.
The fact that $T_k(\pi)$ is finitely generated is proved in the next subsection.
\end{proof}

Using the same idea, we can extend $M_k$ to a $T_k(\pi)$-valued crossed homomorphism on the mapping class group.
For each $\pi$-marked trivalent bordered fatgraph $G$, we consider the composition
$$
\xymatrix{
\Mcg \ar[r]^-{\widetilde{M}_{G,k}} \ar@/_1.3pc/@{-->}[rr]_-{\widetilde{m}_{G,k}} & 
\Bar_3(\pi/\Gamma_{k+1}\pi)/\Img(\partial_4) \ar[r]^-{\SW_3} & T_k(\pi)
}
$$
where the map $\widetilde{M}_{G,k}$ is defined in Corollary \ref{cor:mapping_class_group}.

\begin{corollary}
\label{cor:mapping_class_group_infinitesimal}
For every $k\geq 1$, $\widetilde{m}_{G,k}: \Mcg \to T_k(\pi)$
is an extension of the $k$-th Morita homomorphism to the mapping class group,
with values in a finitely generated free abelian group.
Moreover, $\widetilde{m}_{G,k}$ is a crossed homomorphism whose homology class
$$
\left[\widetilde{m}_{G,k}\right] \in H^1\left(\Mcg;T_k(\pi)\right)
$$
does not depend on the choice of $G$.
\end{corollary}

\noindent
As mentioned in the Introduction, a result similar to Corollary \ref{cor:mapping_class_group_infinitesimal}
is obtained by Day in \cite{Day1,Day2} with a different approach of Malcev Lie algebras.

\subsection{Finite generation}

This subsection is devoted to the proof of the following, which we used in Theorem \ref{th:infinitesimal}.

\begin{lemma}
\label{lem:finite_generation}
The  abelian group $T_k(\pi)$, defined by (\ref{eq:target}), is finitely generated.
\end{lemma}

\begin{proof}
Let $\F_{\leq k}(\gp{x},\gp{y},\gp{z})$ be the free nilpotent group of class $k$ generated by $\{\gp{x},\gp{y},\gp{z}\}$:
this is the $k$-th nilpotent quotient $\F(\gp{x},\gp{y},\gp{z})/\Gamma_{k+1}\F(\gp{x},\gp{y},\gp{z})$
of the free group $\F(\gp{x},\gp{y},\gp{z})$ on three generators.
Let also $\Lie_{\leq k}^\Q(x,y,z)$ be the free nilpotent Lie $\Q$-algebra of class $k$ generated by $\{x,y,z\}$:
we have $\Lie_{\leq k}^\Q(x,y,z) = \Lie^\Q(x,y,z)/\Gamma_{k+1} \Lie^\Q(x,y,z)$
where $\Lie^\Q(x,y,z)$ is the free Lie $\Q$-algebra  on three generators and
$\Gamma_{k+1} \Lie^\Q(x,y,z)$ is the $(k+1)$-st term of its lower central series.
The Malcev Lie algebra $\Malcev\left(\F_{\leq k}(\gp{x},\gp{y},\gp{z})\right)$ can be identified with 
$\Lie_{\leq k}^\Q(x,y,z)$ by setting $x:=\log(\gp{x}),y:=\log(\gp{y})$ and $z:=\log(\gp{z})$.
We \emph{choose} a basis $\{e_i\}_{i\in I}$ of the $\Q$-vector space $\Lie_{\leq k}^\Q(x,y,z)$,
such that each $e_i$ is an iterated bracket $e_i(x,y,z)$ of $x,y,z$ and the indexing set $I$ is totally ordered.
For instance, we can take a Hall basis relative to the set $\{x,y,z\}$ (see \cite{Bourbaki}) 
and, more specifically, we can choose that defined by the lexicographic order:
$$
\begin{array}{l}
x < y < z  \\
\! [x,y] < [x,z] < [y,z]  \\
\! [x,[x,y]] < [x,[x,z]] < [y,[x,y]]< [y,[x,z]]< [y,[y,z]] <  [z,[x,y]] < [z,[x,z]] < [z,[y,z]] \\
\dots \hbox{ etc.}
\end{array}
$$
Thus, there exist some numbers $q_{i_1i_2i_3} \in \Q$ indexed by $i_1<i_2<i_3 \in I$ such that
\begin{equation}
\label{eq:basis}
\SW_3(\gp{x}\vert \gp{y} \vert \gp{z}) = \sum_{i_1<i_2<i_3} 
q_{i_1i_2i_3} \cdot e_{i_1}(x,y,z) \wedge e_{i_2}(x,y,z) \wedge e_{i_3}(x,y,z)\ \in \Lambda^3 \Lie_{\leq k}^\Q(x,y,z).
\end{equation}

For any $\gp{f},\gp{g},\gp{h} \in \pi/\Gamma_{k+1} \pi$, there is a unique group map
$\rho: \F_{\leq k}(\gp{x},\gp{y},\gp{z}) \to \pi/\Gamma_{k+1} \pi$ defined by
$\rho(\gp{x}):=\gp{f}, \rho(\gp{y}) := \gp{g}, \rho(\gp{z}) := \gp{h}$.
An important property of the chain map $\SW$ is its functoriality \cite{SW}, 
which gives the following commutative square:
$$
\xymatrix{
\Bar_3\left(\F_{\leq k}(\gp{x},\gp{y},\gp{z})\right) \ar[r]^-{\SW_3} \ar[d]_-{\Bar_3(\rho)} & 
\Lambda^3 \Lie_{\leq k}^\Q(x,y,z) \ar[d]^-{\Lambda^3 \Malcev(\rho)}\\
\Bar_3\left(\pi/\Gamma_{k+1} \pi\right) \ar[r]_-{\SW_3} & \Lambda^3 \Malcev\left(\pi/\Gamma_{k+1} \pi\right)
}
$$
We deduce from (\ref{eq:basis}) that
\begin{equation}
\label{eq:basis_bis}
\SW_3(\gp{f}\vert \gp{g} \vert \gp{h}) = \sum_{i_1<i_2<i_3} 
q_{i_1i_2i_3} \cdot e_{i_1}(f,g,h) \wedge e_{i_2}(f,g,h) \wedge e_{i_3}(f,g,h)
\end{equation}
where $f:=\log(\gp{f}),g:=\log(\gp{g}), h:=\log(\gp{h})$.

\begin{quote}
\textbf{Claim.}
\emph{The subgroup of $\Malcev(\pi/\Gamma_{k+1} \pi)$ generated by the subset
$$
\log\left(\pi/\Gamma_{k+1} \pi\right) := 
\big\{\log(\gp{g}) \vert \gp{g} \in \pi/\Gamma_{k+1} \pi \big\}
$$
is finitely generated.}
\end{quote}

Assuming this, let $\{f_j\}_{j\in J}$ be a finite generating set for that subgroup of $\Malcev(\pi/\Gamma_{k+1} \pi)$. 
Thus, any element $f=\log(\gp{f})$ of $\log\left(\pi/\Gamma_{k+1} \pi\right)$
can be written as a linear combination with integer coefficients of the $f_j$ (with $j\in J$).
We deduce from  equation (\ref{eq:basis_bis}), which is valid for any $\gp{f},\gp{g},\gp{h} \in \pi/\Gamma_{k+1} \pi$, that
$$
\SW_3\big(\Bar_3(\pi/\Gamma_{k+1}\pi)\big) \subset \frac{1}{Q}
\Lambda^3 \big\langle f_j|j\in J \big\rangle_{\operatorname{Lie}}
$$
where $Q$ is the lowest common denominator of the $q_{i_1 i_2 i_3}$ (with $i_1<i_2<i_3 \in I$)
and $\langle f_j|j\in J\rangle_{\operatorname{Lie}}$ 
is the Lie subring of $\Malcev(\pi/\Gamma_{k+1} \pi)$ generated by the $f_j$ (with $j\in J$).
We conclude that the abelian group $\SW_3\left(\Bar_3(\pi/\Gamma_{k+1}\pi)\right)$
is finitely generated, and the same conclusion applies to $T_k(\pi)$.

It remains to prove the above claim. For this, we fix a basis $\{\gp{z}_1,\dots,\gp{z}_{2g}\}$ of $\pi$ 
and we set  $z_i := \log(\gp{z}_i)$ for all $i=1,\dots,2g$.
Thus, $\Malcev(\pi/\Gamma_{k+1} \pi)$ is identified  
with the free nilpotent Lie algebra $\Lie_{\leq k}^\Q(z_1,\dots,z_{2g})$ of class $k$, 
and $\log(\pi/\Gamma_{k+1} \pi)$ corresponds to the subset 
$$
L := \big\{ \bch(\varepsilon_1 \cdot z_{i_1},\dots, \varepsilon_r \cdot z_{i_r}) \mod \Gamma_{k+1}
\left\vert r\geq 1, \varepsilon_1,\dots,\varepsilon_r =\pm 1, i_1,\dots,i_r = 1,\dots, 2g\right.\big\}
$$
of $\Lie_{\leq k}^\Q(z_1,\dots,z_{2g})$. Here,  
$$
\bch(u_1,\dots,u_r) := \log\left(\exp(u_1) \cdots \exp(u_r)\right) \ \in \widehat{\Lie}^\Q(u_1,\dots,u_r) 
$$
denotes the multivariable Baker--Campbell--Hausdorff series which, before truncation, 
lives in the degree completion of the free Lie algebra. 
Dynkin's formula (see \cite{Bourbaki}) gives
\begin{equation}
\label{eq:Dynkin}
\bch(u_1,\dots,u_r) =
\sum_{l\geq 1} \frac{(-1)^{l-1}}{l}\sum_{P \in \mathcal{P}_{r,l}}
\frac{1}{\displaystyle \big(\sum_{i,j} p_{ij}\big) \cdot \prod_{i,j} p_{ij}!}  
\cdot \left[u_1^{p_{11}} \cdots u_r^{p_{r1}} \cdots u_1^{p_{1l}} \cdots u_r^{p_{rl}} \right]
\end{equation}
where the second sum is over the set of matrices
$$
\mathcal{P}_{r,l} := \left\{P=\left(p_{ij}\right) \in \operatorname{Mat}(r\times l;\Z): 
p_{ij} \geq 0, p_{1j}+ \cdots + p_{rj}>0\right\}
$$
and
$[u_1^{p_{11}} \cdots u_r^{p_{r1}} \cdots u_1^{p_{1l}} \cdots u_r^{p_{rl}}]$
denotes the right-to-left bracketing of the word inside.
For $r,l\geq 1$ and $P\in \mathcal{P}_{r,l}$, we denote by $n(P)$ the number of $i=1,\dots,r$ such that $\sum_{j} p_{ij} >0$, 
and we set $|P| := \sum_{i,j} p_{i,j}$. Since $l\leq |P|$ and $n(P) \leq |P|$, we have
$$
l \cdot\big(\sum_{i,j} p_{ij}\big)  \cdot \prod_{i,j} p_{ij}! \leq l \cdot |P| \cdot \left(|P| !\right)^{l\cdot n(P)}
\leq |P|^2 \cdot  \left(|P| !\right)^{|P|^2}.
$$
It follows that, when the series (\ref{eq:Dynkin}) is truncated up to the order $k$, 
the denominators of its coefficients are bounded uniformly with respect to the number of variables $r$:
let $N_k$ be the least common multiple of those integers.
We deduce that
$$
L \subset \frac{1}{N_k} \cdot \Lie_{\leq k}(z_1,\dots,z_{2g})
$$
where $\Lie_{\leq k}(z_1,\dots,z_{2g})= \Lie(z_1,\dots,z_{2g})/\Gamma_{k+1}\Lie(z_1,\dots,z_{2g})$
is the free nilpotent Lie ring of class $k$ generated by $\{z_1,\dots,z_{2g}\}$.
We conclude that the subgroup of $\Lie_{\leq k}^\Q(z_1,\dots,z_{2g})$ 
spanned by $L$ is finitely generated, which proves the Claim.
\end{proof}

\subsection{Example: the abelian case}

\label{subsec:first_Morita}

Let us apply Theorem \ref{th:infinitesimal} to the case $k=1$.
Thus, we consider the abelian group $H=\pi/\Gamma_2 \pi$.
The first Morita homomorphism is a map
$$
M_1: \Torelli \longrightarrow H_3(H) \simeq \Lambda^3 H
$$
where $\Lambda^3 H$ is identified to $H_3(H)$ in the usual way:
a trivector $h_1 \wedge h_2 \wedge h_3$ corresponds to the
homology class of 
$\sum_{\sigma \in \mathfrak{S}_3} \varepsilon(\sigma) 
\cdot (\gp{h}_{\sigma(1)}\vert \gp{h}_{\sigma(2)} \vert \gp{h}_{\sigma(3)})$.
Here and for clarity, an element of $H$ is denoted by $h$ or by $\gp{h}$
depending on whether $H$ is regarded as a $\Z$-module or as a group.

\begin{corollary}
\label{cor:M_1}
There is a groupoid homomorphism $\widetilde{m}_1: \Pt/\Torelli \to \frac{1}{6} \Lambda^3H$ defined by
\begin{equation}
\label{eq:M_1_formula}
\widetilde{m}_1\left(
\begin{array}{c}
{\labellist \small \hair 1pt
\pinlabel {$\displaystyle \mathop{\longrightarrow}^W$} at 157 30
\pinlabel {$h_1$} [tr] at 10 58
\pinlabel {$h_2$} [br] at 5 5
\pinlabel {$h_3$} [bl] at 103 8
\pinlabel {$h_1$} [tr] at 219 58
\pinlabel {$h_2$} [br] at 217 7
\pinlabel {$h_3$} [bl] at 307 5
\endlabellist
\includegraphics[scale=0.5]{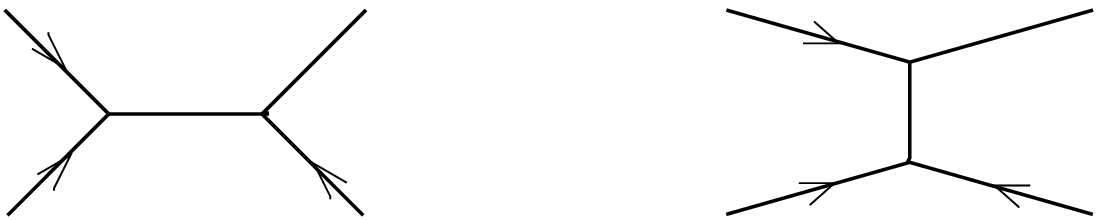}}
\end{array}
\right) = \frac{1}{6} \cdot h_1 \wedge h_2 \wedge h_3
\end{equation}
on each Whitehead move $W$ between $H$-marked trivalent bordered fatgraphs.
This is an extension of $M_1$ to the Torelli groupoid;
more precisely, the following square commutes for every $\pi$-marked trivalent bordered fatgraph $G$:
\begin{equation}
\label{eq:M_1}
\xymatrix{
\pi_1\left(\MFat/\Torelli,\{G\}\right) \ar@{^{(}->}[d] \ar@{=}[r]^-G & \Torelli \ar[r]^-{M_1} &  \Lambda^3 H \ar@{^{(}->}[d]\\
\pi_1^{\operatorname{cell}}\left(\MFat/\Torelli\right) \ar@{=}[r] & \Pt/\Torelli  \ar[r]_-{\widetilde{m}_1} & \frac{1}{6} \Lambda^3H.
}
\end{equation}
\end{corollary}

\noindent
Recall that $M_1$ coincides with (the opposite of) the first Johnson homomorphism \cite{Morita_abelian}.
Therefore $\widetilde{m}_1$ is essentially the same as Morita and Penner's extension of this homomorphism to the Torelli groupoid \cite{MP}.
See also \S \ref{subsec:first_Johnson}.

\begin{proof}[Proof of Corollary \ref{cor:M_1}]
The abelian Lie algebra $\Malcev(H)$ can be identified with $H \otimes \Q$:
for all $\gp{h} \in H$, $\log(\gp{h})$ corresponds to $h\otimes 1$.
According to Proposition \ref{prop:abelian_case}, we have
$$
\SW_3: \Bar_3(H) \longrightarrow \Lambda^3 \Malcev(H) \simeq \Lambda^3 H\otimes \Q, \ 
(\gp{h}_1 \vert \gp{h}_2 \vert \gp{h}_3) \longmapsto \frac{1}{6} h_1 \wedge h_2 \wedge h_3.
$$
It follows that
$$
T_1(\pi) \simeq \frac{1}{6} \Lambda^3 H \ \subset \Lambda^3 H \otimes \Q.
$$
The map $\Lambda^3 H \simeq H_3(H) \stackrel{\SW_3}{\to} T_1(\pi) \simeq \frac{1}{6} \Lambda^3 H$
is the inclusion $\Lambda^3 H \subset \frac{1}{6} \Lambda^3 H$.
Therefore the map $\widetilde{m}_1$ produced by Theorem \ref{th:infinitesimal} satisfies diagram (\ref{eq:M_1}). 
Formula (\ref{eq:M_1_formula}) for $\widetilde{m}_1(W)$  follows from the definition of the $3$-chain $T_W$ given in Lemma \ref{lem:M}.
\end{proof}

\section{Infinitesimal extensions of Johnson homomorphisms}

We have defined in \S \ref{sec:infinitesimal} an extension of $M_k$ to the Ptolemy groupoid 
$$
\widetilde{m}_k: \Pt/\Mcg[k] \longrightarrow 
T_k(\pi) \subset \frac{\Lambda^3 \Malcev(\pi/\Gamma_{k+1}\pi)}{\Img(\partial_4)}
$$
whose target $T_k(\pi)$ is a finitely generated abelian group. 
In this section, we shall derive from $\widetilde{m}_k$ an extension
of the $k$-th Johnson homomorphism to the Ptolemy groupoid.

\subsection{Construction}

We start by recalling the precise definition of Johnson's homomorphisms \cite{Johnson_abelian,Johnson_survey,Morita_abelian}.
For all $k\geq 1$, let $\rho_k: \Mcg \to \Aut(\pi/\Gamma_{k+1}\pi)$ be the canonical homomorphism: its kernel is, by definition,
the $k$-th term $\Mcg[k]$ of the Johnson filtration. There is a short exact sequence
$$
1 \to \Hom\left(\pi/\Gamma_2 \pi, \Gamma_{k+1} \pi/ \Gamma_{k+2} \pi\right)
\to \Aut(\pi/\Gamma_{k+2} \pi) \to \Aut(\pi/\Gamma_{k+1} \pi)
$$
where  a group homomorphism $t: \pi/\Gamma_2 \pi \to \Gamma_{k+1}\pi/\Gamma_{k+2} \pi$
goes to the automorphism of $\pi/\Gamma_{k+2} \pi$ defined by $\{\gp{x}\} \mapsto \{\gp{x} \cdot t(\{\gp{x}\})\}$.
Thus, the map $\rho_{k+1}$ restricts to a  homomorphism
$$
\tau_k: \Mcg[k] \longrightarrow \Hom(\pi/\Gamma_2 \pi,\Gamma_{k+1} \pi/\Gamma_{k+2} \pi)
\simeq  \Hom(H, \Lie_{k+1}(H)) \simeq H \otimes \Lie_{k+1}(H)
$$
which is called the \emph{$k$-th Johnson homomorphism}.
Here, $\Lie_n(H)$ denotes the degree $n$ part of the free Lie ring over $H$,
which is identified in the canonical way with $\Gamma_n \pi/\Gamma_{n+1} \pi$ \cite{Bourbaki}; 
the second isomorphism in the definition of $\tau_k$ is induced by the duality $H\simeq H^*$ 
defined by $h \mapsto \omega(h,-)$, where $\omega:H \times H \to \Z$ is the intersection pairing.\footnote{If one identifies
$H$ with $H^*$ by $h\mapsto \omega(-,h)$, then the definition of $\tau_k$ differs by a minus sign.
This convention seems to be used in Johnson's and Morita's papers.}

Morita showed in \cite{Morita_abelian} that $M_k$ determines $\tau_k$ in an explicit way.
Similarly, the ``infinitesimal'' version $m_k$ determines $\tau_k$ as follows.
Consider the central extension of Lie algebras
\begin{equation}
\label{eq:extension}
0 \longrightarrow \Lie_{k+1}(H_\Q) \longrightarrow 
\Malcev(\pi/\Gamma_{k+2}\pi) \longrightarrow \Malcev(\pi/\Gamma_{k+1} \pi) \longrightarrow 1
\end{equation}
whose first map is the composition
$$
\Lie_{k+1}(H_\Q) \stackrel{\simeq}{\longrightarrow} 
(\Gamma_{k+1}\pi/ \Gamma_{k+2} \pi) \otimes \Q 
\stackrel{\log \otimes \Q}{\longrightarrow} \Malcev(\pi/\Gamma_{k+2} \pi).
$$
The differential $d^2_{p,q}: E^2_{p,q} \to E^2_{p-2,q+1}$
of the second stage of the Hochschild--Serre spectral sequence associated to (\ref{eq:extension})
$$
E^2_{p,q} \simeq H_p\left(\Malcev(\pi/\Gamma_{k+1} \pi);\Q\right) \otimes \Lambda^q \Lie_{k+1}(H_\Q)
$$
gives for $p=3$ and $q=0$ a homomorphism 
$$
d^2_{3,0}: H_3\left(\Malcev(\pi/\Gamma_{k+1} \pi);\Q\right) \to H_\Q \otimes \Lie_{k+1}(H_\Q).
$$
The following identity is proved in \cite{Massuyeau}:
\begin{equation}
\label{eq:Morita_to_Johnson}
\forall f\in \Mcg[k], \ -d^2_{3,0} \circ m_k (f) =\tau_k(f).
\end{equation}

To derive from $\widetilde{m}_k$ an extension of $\tau_k$ to the $k$-th Torelli groupoid,
we need an extension of the differential $d^2_{3,0}$ introduced by Day in \cite{Day1}. 
His ``extended differential'' is the map
$$
\widetilde{d}^2_{3,0}: \frac{\Lambda^3 \Malcev(\pi/\Gamma_{k+1} \pi)}{\Img(\partial_4)}
\longrightarrow \frac{\Lambda^2 \Malcev(\pi/\Gamma_{k+2}\pi)}{\Gamma_2\Malcev(\pi/\Gamma_{k+2}\pi)\wedge \Gamma_{k+1}\Malcev(\pi/\Gamma_{k+2}\pi)}
$$
defined by the formula
$$
\widetilde{d}^2_{3,0}\left(\{x\}\right) := \{ \partial_3(\widetilde{x})\}
$$
where $\widetilde{x}\in \Lambda^3 \Malcev(\pi/\Gamma_{k+2}\pi)$ is a lift of $x \in \Lambda^3 \Malcev(\pi/\Gamma_{k+1}\pi)$
and $\partial_3$ denotes the boundary in the Koszul complex. Thus, we have the commutative diagram
$$
\xymatrix{
H_3\left(\Malcev(\pi/\Gamma_{k+1} \pi);\Q\right) \ar@{^{(}->}[r] \ar[d]_-{d^2_{3,0}} & 
\frac{\Lambda^3 \Malcev(\pi/\Gamma_{k+1} \pi)}{\Img(\partial_4)} \ar[d]^-{\widetilde{d}^2_{3,0}}\\
H_\Q \otimes \Lie_{k+1}(H_\Q)\ \ar@{>->}[r]_-\iota
& \frac{\Lambda^2 \Malcev(\pi/\Gamma_{k+2}\pi)}{\Gamma_2\Malcev(\pi/\Gamma_{k+2}\pi)\wedge \Gamma_{k+1}\Malcev(\pi/\Gamma_{k+2}\pi)}
}
$$
where the map $\iota$ is induced by the isomorphism
$$
H_\Q \otimes \Lie_{k+1}(H_\Q) \simeq
\frac{\Malcev(\pi/\Gamma_{k+2}\pi)}{\Gamma_2\Malcev(\pi/\Gamma_{k+2}\pi)} \otimes \Lie_{k+1}(H_\Q)
\simeq \frac{\Malcev(\pi/\Gamma_{k+2}\pi) \otimes  
\Gamma_{k+1}\Malcev(\pi/\Gamma_{k+2}\pi)}{\Gamma_2\Malcev(\pi/\Gamma_{k+2}\pi) \otimes  \Gamma_{k+1}\Malcev(\pi/\Gamma_{k+2}\pi)}.
$$
The extended differential allows us to define the subgroup
$$
U_k(\pi) := \widetilde{d}^2_{3,0}\left(T_k(\pi)\right) \subset 
\frac{\Lambda^2 \Malcev(\pi/\Gamma_{k+2}\pi)}{\Gamma_2\Malcev(\pi/\Gamma_{k+2}\pi)\wedge \Gamma_{k+1}\Malcev(\pi/\Gamma_{k+2}\pi)}.
$$
It follows from Lemma \ref{lem:finite_generation} that $U_k(\pi)$ is  finitely generated.

\begin{theorem}
\label{th:Johnson}
The homomorphism $\widetilde{\tau}_k$ defined by
$$
\xymatrix{
\Pt/\Mcg[k] \ar[r]^-{-\widetilde{m}_k} \ar@/_1.2pc/@{-->}[rr]_-{\widetilde{\tau}_k} & 
T_k(\pi) \ar[r]^-{\widetilde{d}^2_{3,0}}&  U_k(\pi)
}
$$
is a groupoid extension of $\tau_k: \Mcg[k] \to \Img(\tau_k)$ 
whose target $U_k(\pi)$ is a finitely generated free abelian group.  
\end{theorem}

\begin{proof}
Let $f\in \Mcg[k]$ and, given a $\pi$-marked trivalent bordered fatgraph $G$, represent $f$ as a sequence of Whitehead moves 
$(G=G_1 \overset{W_1}{\to} G_2 \overset{W_2}{\to} \cdots \overset{W_r}{\to} G_{r+1}=f(G))$. 
According to formula (\ref{eq:Morita_to_Johnson}), we have 
\begin{equation}
\label{eq:Johnson}
\iota \tau_k(f) = - \iota  d^2_{3,0}  m_k(f) = -  \widetilde{d}^2_{3,0} \widetilde{m}_k\left(
(G_1 \overset{W_1}{\to} G_2 \overset{W_2}{\to} \cdots \overset{W_r}{\to} G_{r+1})\mod \Mcg[k]\right).
\end{equation}
Since $\widetilde{m}_k= \SW_3 \circ \widetilde{M}_k$, we deduce that
$\iota \tau_k(f)$ lives in the image of $\widetilde{d}^2_{3,0}\circ \SW_3$, \ie in the group $U_k(\pi)$.
Therefore, $\Img(\tau_k) \subset H \otimes \Lie_{k+1}(H) \subset H_\Q \otimes \Lie_{k+1}(H_\Q)$ is sent by $\iota$
into $U_k(\pi)$. Equation (\ref{eq:Johnson}) shows that the diagram
\begin{equation}
\label{eq:Johnson_extension}
\newdir{ >}{{}*!/-5pt/\dir{>}}
\xymatrix{
\pi_1\left(\MFat/\Mcg[k],\{G\}\right) \ar@{^{(}->}[d] \ar@{=}[r]^-G & \Mcg[k] \ar[r]^-{\tau_k} 
& \Img(\tau_k)  \ar@{ >->}[d]^-{\iota} \\
\pi_1^{\operatorname{cell}}\left(\MFat/\Mcg[k]\right) \ar@{=}[r]  & \Pt/\Mcg[k] \ar[r]_-{\widetilde{\tau}_k}
& U_k(\pi)
} 
\end{equation}
commutes for every $\pi$-marked trivalent bordered fatgraph $G$.
\end{proof}

The next proposition shows that the groupoid extension $\widetilde{\tau}_k$ can be computed from the Suslin--Wodzicki chain map in degree $2$
(whereas the groupoid extension $\widetilde{m}_k$ needs the same chain map in degree $3$).

\begin{proposition}
\label{prop:bivector}
The value of $\widetilde{\tau}_k$ on a Whitehead move $W:G\to G'$ is represented by the bivector
$$
-s\cdot \SW_2\!\Big((\{\gp{h}\}|\{\gp{g}\})-(\{\gp{kh}\}|\{\gp{g}\})+(\{\gp{k}\}|\{\gp{hg}\})-(\{\gp{k}\}|\{\gp{h}\})\Big)
\in \Lambda^2 \Malcev(\pi/\Gamma_{k+2}\pi)
$$
where $\{\gp{g}\},\{\gp{h}\},\{\gp{k}\}\in \pi/\Gamma_{k+2}\pi$ and the sign $s$ are given on Figure \ref{fig:T_W}.
\end{proposition} 

\begin{proof}
We have
$$
\widetilde{\tau}_k\left(W\mod \Mcg[k]\right) = -\widetilde{d}^2_{3,0} \widetilde{m}_k\left(W\mod \Mcg[k]\right)
= -\widetilde{d}^2_{3,0} \SW_3 \widetilde{M}_k\left(W\mod \Mcg[k]\right).
$$
Therefore, by definition of the groupoid map $\widetilde{M}_k$, we get
$$
\widetilde{\tau}_k\left(W\mod \Mcg[k]\right) = -  \widetilde{d}^2_{3,0} \SW_3\big(s\cdot (\{\gp{k}\}|\{\gp{h}\}|\{\gp{g}\})\big)
$$
where $\{\gp{g}\},\{\gp{h}\},\{\gp{k}\} \in \pi/\Gamma_{k+1}\pi$
are the classes of the elements $\gp{g},\gp{h},\gp{k} \in \pi$ shown on Figure \ref{fig:T_W}.
By functoriality of the  Suslin--Wodzicki chain map,  
$\SW_3\big(s\cdot (\{\gp{k}\}|\{\gp{h}\}|\{\gp{g}\})\big) \in \Lambda^3\Malcev(\pi/\Gamma_{k+1}\pi)$ can be lifted to
a $3$-chain  $\SW_3\big(s\cdot (\{\gp{k}\}|\{\gp{h}\}|\{\gp{g}\})\big) \in \Lambda^3\Malcev(\pi/\Gamma_{k+2}\pi)$, 
where  $\{\gp{g}\},\{\gp{h}\},\{\gp{k}\}$ now denote elements of $\pi/\Gamma_{k+2}\pi$. We deduce that
$$
\widetilde{\tau}_k\left(W\mod \Mcg[k]\right) = \left\{- \partial_3 \SW_3\big(s\cdot (\{\gp{k}\}|\{\gp{h}\}|\{\gp{g}\})\big)\right\}
= \left\{- s\cdot \SW_2\partial_3\big(\{\gp{k}\}|\{\gp{h}\}|\{\gp{g}\}\big)\right\}
$$
and the conclusion follows.
\end{proof}

Finally, we obtain  an extension of $\tau_k$ to an $U_k(\pi)$-valued crossed homomorphism on the full mapping class group.
Indeed, for every $\pi$-marked trivalent bordered fatgraph $G$, we can consider the composition
$$
\xymatrix{
\Mcg \ar[r]^-{-\widetilde{m}_{G,k}} \ar@/_1.2pc/@{-->}[rr]_-{\widetilde{\tau}_{G,k}} & 
T_k(\pi) \ar[r]^-{\widetilde{d}^2_{3,0}} & U_k(\pi)
}
$$
where $\widetilde{m}_{G,k}$ is defined in Corollary \ref{cor:mapping_class_group_infinitesimal}.

\begin{corollary}
\label{cor:mapping_class_group_Johnson}
For all $k\geq 1$, $\widetilde{\tau}_{G,k}: \Mcg \to U_k(\pi)$
is an extension of the $k$-th Johnson homomorphism with values in a finitely generated free abelian group.
Moreover, $\widetilde{\tau}_{G,k}$ is a crossed homomorphism whose homology class
$$
\left[\widetilde{\tau}_{G,k}\right] \in H^1\left(\Mcg;U_k(\pi)\right)
$$
does not depend on the choice of $G$.
\end{corollary}

\noindent
As mentioned in the Introduction,
similar extensions of $\tau_k$ to the mapping class group are obtained by Day in \cite{Day1,Day2}.

\subsection{Example: the abelian case}

\label{subsec:first_Johnson}

Let us consider the case $k=1$.
The  group $\Lambda^3H$ embeds into $H \otimes \Lie_2(H)$ by the map
$k \wedge h \wedge g \mapsto k \otimes [h,g] + g \otimes [k,h] + h \otimes [g,k]$.
The first Johnson homomorphism takes values in that subgroup \cite{Johnson_abelian}:
$$
\tau_1: \Torelli \longrightarrow \Lambda^3 H.
$$
According to Proposition \ref{prop:bivector}, the groupoid extension $\widetilde{\tau}_1$ of $\tau_1$
can be computed from the map $\SW_2$ for the nilpotent group $\pi/\Gamma_3 \pi$.

\begin{lemma}
\label{lem:degree_2}
Let $G$ be a finitely generated torsion-free nilpotent group of class $2$.
Then the chain map $\SW$ is given in degree $2$ by
$$
 \SW_2(\gp{g}_1|\gp{g}_2) = \frac{1}{2}\cdot \log(\gp{g}_1) \wedge \log(\gp{g}_2) 
+ \frac{1}{12} \cdot \big( \log(\gp{g}_1) - \log(\gp{g}_2)\big) \wedge [\log(\gp{g}_1),\log(\gp{g}_2)] 
$$
for all $\gp{g}_1,\gp{g}_2 \in G$.
\end{lemma}

\begin{proof}
Let $\F_{\leq k}(\gp{x},\gp{y})$ be the nilpotent group of class $k$ freely generated by $\{\gp{x},\gp{y}\}$.
By functoriality of $\SW$, it is enough to prove the lemma for $G=\F_{\leq 2}(\gp{x},\gp{y})$ and, for this,
we shall work with $\F_{\leq 3}(\gp{x},\gp{y})$.
The Lie algebra $\Malcev(\F_{\leq 3}(\gp{x},\gp{y}))$ is nilpotent of class $3$ freely generated by $\{x,y\}$, where $x := \log(\gp{x})$ and $y := \log(\gp{y})$.
Therefore $\left(x,y,[x,y],[x,[x,y]],[y,[x,y]]\right)$ is a basis of $\Malcev(\F_{\leq 3}(\gp{x},\gp{y}))$ as a $\Q$-vector space,
and it also induces a basis of $\Lambda^2 \Malcev(\F_{\leq 3}(\gp{x},\gp{y}))$.
In this basis, we can write
$$
\SW_2\left(\gp{x}|\gp{y}\right) = a\cdot x\wedge y + b \cdot x \wedge[x,y] + c \cdot y \wedge [x,y]  + \hbox{etc} \ \in \Lambda^2\Malcev(\F_{\leq 3}(\gp{x},\gp{y}))
$$
where $a,b,c \in \Q$ and the ``etc'' part stands for basis elements of degree at least $4$.
Consequently, we have
$$
\partial_2 \SW_2\left(\gp{x}|\gp{y}\right) =  - a\cdot [x,y] - b \cdot [x,[x,y]] - c \cdot [y, [x,y]]  + 0 \ \in \Malcev(\F_{\leq 3}(\gp{x},\gp{y})).
$$
We also deduce from Corollary \ref{cor:SW_1} and the Baker--Campbell--Hausdorff formula that
$$
\SW_1 \partial_2\left(\gp{x}|\gp{y}\right) = \log(\gp{y} -\gp{xy} +\gp{x}) 
= - \frac{1}{2}\cdot [x,y] - \frac{1}{12}\cdot[x,[x,y]] + \frac{1}{12}\cdot[y,[x,y]].
$$
We conclude that $a=1/2$ and $b=1/12=-c$. 
Since $\SW$ is functorial, this proves the lemma for $G=\F_{\leq 2}(\gp{x},\gp{y})$.
\end{proof}

We shall then deduce the following from  Theorem \ref{th:Johnson}.

\begin{corollary}
\label{cor:tau_1}
There is a groupoid homomorphism $\widetilde{\tau}_1: \Pt/\Torelli \to \frac{1}{6} \Lambda^3H$ defined by
\begin{equation}
\label{eq:tau_1_formula}
\widetilde{\tau}_1\left(
\begin{array}{c}
{\labellist \small \hair 1pt
\pinlabel {$\displaystyle \mathop{\longrightarrow}^W$} at 157 30
\pinlabel {$h_1$} [tr] at 10 58
\pinlabel {$h_2$} [br] at 5 5
\pinlabel {$h_3$} [bl] at 103 7
\pinlabel {$h_1$} [tr] at 219 58
\pinlabel {$h_2$} [br] at 219 8
\pinlabel {$h_3$} [bl] at 307 5
\endlabellist
\includegraphics[scale=0.5]{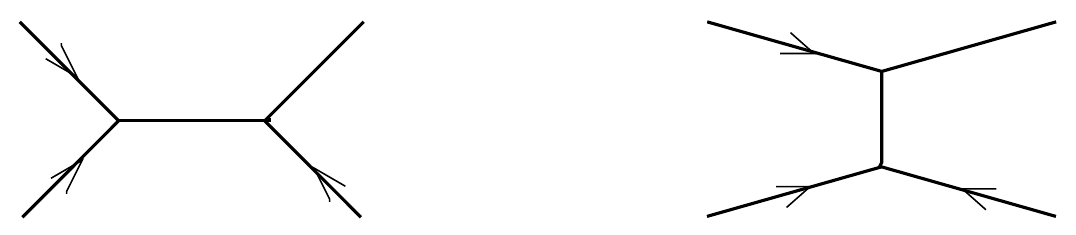}}
\end{array}
\right) = -\frac{1}{6} \cdot h_1 \wedge h_2 \wedge h_3
\end{equation}
on each Whitehead move $W$ between $H$-marked trivalent bordered fatgraphs, 
and this is an extension of $\tau_1$ to the Torelli groupoid.
\end{corollary}

\noindent
We deduce from (\ref{eq:M_1_formula}) and (\ref{eq:tau_1_formula}) that $\widetilde{\tau}_1=-\widetilde{m}_1$,
which is compatible with the fact that $m_1$ is equal to $-\tau_1$ by (\ref{eq:Morita_to_Johnson}).
The map $-6 \widetilde{\tau}_1$ coincides with Morita and Penner's extension of $-6 \tau_1$ to the Torelli groupoid \cite{MP}.
Besides, Corollary \ref{cor:mapping_class_group_Johnson} produces (for each $\pi$-marked trivalent bordered fatgraph $G$)
an extension $\widetilde{\tau}_{G,1}:\Mcg \to \frac{1}{6} \Lambda^3H$ of $\tau_1$ to the mapping class group. 
We have
$$
H^1\Big(\Mcg;\frac{1}{2}\Lambda^3 H\Big) \longrightarrow H^1\Big(\Mcg;\frac{1}{6}\Lambda^3 H\Big), \
\left[\! -\widetilde{k}\right] \longmapsto \left[\widetilde{\tau}_{G,1}\right]
$$
where $\widetilde{k}$ is Morita's extension of $-\tau_1$ to the mapping class group \cite{Morita_extension}.
(This follows from the fact that $H^1\left(\Sp(2g;\Z);\frac{1}{6 }\Lambda^3 H\right) \simeq H^1\left(\Sp(2g;\Z);\Lambda^3 H\right)$
is trivial, as proved by Morita \cite{Morita_extension}.)

\begin{proof}[Proof of Corollary \ref{cor:tau_1}]
Let $\{\gp{g}\}, \{\gp{h}\}, \{\gp{k}\} \in \pi/\Gamma_3 \pi$,
whose logarithms are denoted by $g,h,k \in \Malcev(\pi/\Gamma_3 \pi)$ respectively.
Using Lemma \ref{lem:degree_2}, we obtain
the following identities in the quotient space $\Lambda^2 \Malcev(\pi/\Gamma_3 \pi) /  \Lambda^2 \Gamma_2 \Malcev(\pi/\Gamma_3 \pi)$:
\begin{eqnarray*}
\! \! \! \! \! \! &&\SW_2\big((\{\gp{h}\}|\{\gp{g}\})-(\{\gp{kh}\}|\{\gp{g}\})+(\{\gp{k}\}|\{\gp{hg}\})-(\{\gp{k}\}|\{\gp{h}\})\big)\\
\! \! \!  \! \! \! &=& \frac{1}{2} h \wedge g + \frac{1}{12}(h-g)\wedge [h,g] - \frac{1}{2} \left(k+h+\frac{1}{2}[k,h]\right)\wedge g - \frac{1}{12}(k+h-g)\wedge [k+h,g] +\\
\! \! \! \! \! \! &&  \frac{1}{2} k \wedge \left(h+g +\frac{1}{2}[h,g]\right) + \frac{1}{12}(k-(h+g)) \wedge [k,h+g] - \frac{1}{2} k \wedge h - \frac{1}{12}(k-h)\wedge [k,h]\\
\! \! \! \! \! \! &=&  \frac{1}{6}h\wedge [g,k]  + \frac{1}{6} k \wedge [h,g]   + \frac{1}{6}g \wedge [k,h].
\end{eqnarray*}
This shows that $U_1(\pi)$ is the image of $\frac{1}{6} \Lambda^3 H$ by the embedding
$$
\xymatrix{
\frac{1}{6} \Lambda^3 H\ \ar@{>->}[r] & H_\Q \otimes \Lie_2(H_\Q)\ \ar@{>->}[r]^\iota &  
\frac{\Lambda^2 \Malcev(\pi/\Gamma_3 \pi)}{\Lambda^2 \Gamma_2 \Malcev(\pi/\Gamma_3 \pi)}.
}
$$
Let $\widetilde{\tau}_1$ be the groupoid extension of $\tau_1$ produced by Theorem \ref{th:Johnson}.
Then formula (\ref{eq:tau_1_formula}) for $\widetilde{\tau}_1(W)$ is  deduced from Proposition \ref{prop:bivector}. 
\end{proof}

\appendix

\section{Review of Malcev Lie algebras and their homology}

Let $G$ be a group. An efficient way to define the Malcev Lie algebra of $G$ is from its group algebra $\Q[G]$.
This construction, initiated by Jennings \cite{Jennings}, is developed by Quillen in \cite{Quillen} to which we refer for full details.
We denote by $I$ the augmentation ideal of $\Q[G]$. The $I$-adic completion of $\Q[G]$
$$
\widehat{\Q}[G] := \varprojlim_{k} \Q[G]/ I^k
$$
equipped with the filtration 
$$
\widehat{I^j} := \varprojlim_{k\geq j} I^j/ I^k, \quad \forall j\geq 0
$$
is a complete Hopf algebra in the sense of Quillen \cite{Quillen}. Let $\widehat \Delta$ be the coproduct. 

\begin{definition}
The \emph{Malcev Lie algebra} of $G$ is the Lie algebra of primitive elements 
$$
\Malcev(G) :=  \Prim(\widehat{\Q}[G]) = 
\big\{x \in \widehat{\Q}[G]: \widehat\Delta(x)= x \widehat{\otimes} 1 + 1 \widehat{\otimes} x \big\}.
$$
\end{definition}

\noindent
Equivalently, one can consider the \emph{Malcev completion} of $G$, defined as the group of group-like elements:
$$
\MalcevCompletion(G) := \GLike(\widehat{\Q}[G]) = 
\big\{x \in \widehat{\Q}[G]: \widehat\Delta(x)= x \widehat{\otimes} x, x\neq 0 \big\}.
$$
The Malcev completion and the Malcev Lie algebra of a group $G$ are
in one-to-one correspondence via the exponential and logarithmic series:
$$
\MalcevCompletion(G) \subset  1 + \widehat{I} 
\xymatrix{
\ar@/^0.5pc/[rr]^-\log_-\simeq &&  \ar@/^0.5pc/[ll]^-\exp
}
\widehat{I}   \supset \Malcev(G).
$$
The inclusion $G \subset \Q[G]$ induces a canonical map $G \to \MalcevCompletion(G)$. 
It is injective if, and only if, $G$ is residually torsion-free nilpotent \cite{Malcev_embedding,Jennings}.
In such a case, we  regard $G$ as a subgroup of $\MalcevCompletion(G)$. 
For example, free groups and free nilpotent groups are residually torsion-free nilpotent.\\

The homology of  groups is related to the homology of their Malcev Lie algebras through the following statement.

\begin{theorem}[Pickel \cite{Pickel}, Suslin--Wodzicki \cite{SW}]
\label{th:PSW}
Let $G$ be a finitely generated torsion-free nilpotent group. 
Then there is a canonical isomorphism
$$
\P: H_*\left(G;\Q\right) \longrightarrow H_*\left(\Malcev(G);\Q\right)
$$
which is induced by a canonical chain map 
$$
\SW: \Bar_*(G) \otimes \Q \longrightarrow \Lambda^* \Malcev(G)
$$
between the bar complex of $G$ with $\Q$-coefficients and the Koszul complex of $\Malcev(G)$.
Moreover, $\SW$ and $\P$ are functorial in $G$.
\end{theorem}

\begin{proof}[Sketch of the proof]
The isomorphism $H_*\left(G;\Q\right) \simeq H_*\left(\Malcev(G);\Q\right)$
is due to Pickel \cite{Pickel}, who proceeds in the following manner.
First, he proves that $\widehat{\Q}[G]$ is flat as a $\Q[G]$-module
and that, similarly, $\Uhat(\Malcev(G))$ is flat as an $\U(\Malcev(G))$-module.
(Here, $\Uhat(\Malcev(G))$ is the completion of the enveloping algebra  $\U(\Malcev(G))$
with respect to powers of $\Malcev(G)$.)
Next, he deduces from \cite{Jennings} that the inclusion 
$\Malcev(G) \subset \widehat{\Q}[G]$ induces an algebra isomorphism
$$
\Uhat(\Malcev(G)) \simeq \widehat{\Q}[G].
$$
Finally, he considers, for all $n\geq 1$, the following sequence of isomorphisms:
$$
\xymatrix @!0 @R=1.1cm @C=5cm  {
{\Tor_n^{\Q[G]}(\Q,\Q)} \ar[r]^-\simeq &
{\Tor_n^{\widehat{\Q}[G]}(\Q,\Q) \simeq   \Tor_n^{\Uhat(\Malcev(G))}(\Q,\Q)} &
 {\Tor_n^{\U(\Malcev(G))}(\Q,\Q)} \ar[l]_-\simeq  \\
{H_n(G;\Q)}  \ar@{=}[u] \ar@{-->}_-{\P} [rr] & & {H_n(\Malcev(G);\Q).} \ar@{=}[u]
}
$$

A chain map $\SW$ inducing $\P$ in homology is defined by Suslin and Wodzicki in \cite{SW} as follows.
As a free resolution of $\Q$ as a $\Q[G]$-module, we take the bar resolution:
$$
\xymatrix{
\cdots \ar[r] & B_2 \ar[r]^-{\partial_2} & B_1 \ar[r]^-{\partial_1} & B_0 \ar[r]^-{\varepsilon} & \Q \ar[r] & 0 
}
$$
where $B_n = \Q[G] \cdot G^{\times n}$, $\varepsilon$ is the augmentation of $\Q[G]$ and
$$
\partial_n\left(\gp{g}_1| \cdots | \gp{g}_n  \right) = \gp{g}_1 \cdot (\gp{g}_2 | \cdots | \gp{g}_n)
+ \sum_{i=1}^{n-1} (-1)^i \cdot (\gp{g}_1 | \cdots |\gp{g}_i \gp{g}_{i+1}| \cdots |\gp{g}_n)
+ (-1)^n \cdot (\gp{g}_1 |\cdots| \gp{g}_{n-1}). 
$$
As  a free resolution of $\Q$ as an  $\U(\Malcev(G))$-module, we take the Koszul resolution:
$$
\xymatrix{
\cdots \ar[r] & K_2 \ar[r]^-{\partial_2} & K_1 \ar[r]^-{\partial_1} & K_0 \ar[r]^-{\eta} & \Q \ar[r] & 0 
}
$$
where $K_n = \U(\Malcev(G)) \otimes \Lambda^n \Malcev(G)$,
$\eta$ is the augmentation of $\U(\Malcev(G))$ and
\begin{eqnarray*}
\partial_n\left(1 \otimes g_1 \wedge \cdots \wedge g_n \right) &=& 
\sum_{i=1}^n (-1)^{i+1} g_i \otimes g_1 \wedge \cdots \widehat{g_i} \cdots \wedge g_n \\
&& + \sum_{1 \leq i<j \leq n} (-1)^{i+j} \otimes [g_i,g_j] 
\wedge g_1 \wedge \cdots \widehat{g_i} \cdots \widehat{g_j} \cdots \wedge g_n.
\end{eqnarray*}
By tensoring and using that  $\widehat{\Q}[G]$ is flat as a $\Q[G]$-module,   
we get a free resolution of $\Q$ as a $\widehat{\Q}[G]$-module:
$$
\widehat{\Q}[G] \otimes_{\Q[G]} B_* \longrightarrow \Q \longrightarrow 0.
$$
Similarly, we obtain a free resolution of $\Q$ as an $\widehat{\U}(\Malcev(G))$-module:
$$
\widehat{\U}(\Malcev(G)) \otimes_{\U(\Malcev(G))} K_* \longrightarrow \Q \longrightarrow 0.
$$
Consequently, there exists a homotopy equivalence  
\begin{equation}
\label{eq:chain_map}
f: \widehat{\Q}[G] \otimes_{\Q[G]} B_* \longrightarrow \widehat{\U}(\Malcev(G)) \otimes_{\U(\Malcev(G))} K_*
\end{equation}
of chain complexes over the ring $\widehat{\Q}[G] = \Uhat(\Malcev(G))$.
Such a homotopy equivalence is unique up to homotopy.
To construct an explicit one in \cite{SW}, Suslin and Wodzicki first define a canonical contracting  homotopy
\begin{equation}
\label{eq:contracting_homotopy}
s:  \widehat{\U}(\Malcev(G)) \otimes_{\U(\Malcev(G))} K_* \longrightarrow  \widehat{\U}(\Malcev(G)) \otimes_{\U(\Malcev(G))} K_{*+1}
\end{equation}
using the Poincar\'e--Birkhoff--Witt isomorphism. 
Next, they define $f$ to be the identification $\widehat{\Q}[G] = \Uhat(\Malcev(G))$ in degree $0$ and they use  the  inductive formula 
$$
\forall r \in \widehat{\Q}[G] = \Uhat(\Malcev(G)), \forall \gp{g_1}, \dots, \gp{g}_n \in G, \
f_n\left(r\cdot (\gp{g_1}\vert \cdots \vert \gp{g}_n)\right) = r \cdot s_{n-1} f_{n-1} \partial_{n}  (\gp{g_1}\vert \cdots \vert \gp{g}_n)
$$
in degree $n>0$. (This is the usual way of constructing a chain map from a contracting homotopy \cite[Proposition XI.5.2]{CE}.)
Since the bar complex with $\Q$-coefficients is 
$$
\Bar_*(G) \otimes \Q = \Q \otimes_{\Q[G]} B_* = \Q \otimes_{\widehat{\Q}[G]} \left(\widehat{\Q}[G] \otimes_{\Q[G]} B_*\right)
$$
and since the Koszul complex is 
$$
\Lambda^* \Malcev(G) = \Q \otimes_{\U(\Malcev(G))} K_* = 
\Q \otimes_{\Uhat(\Malcev(G))} \left(\Uhat(\Malcev(G))\otimes_{\U(\Malcev(G))} K_*\right),
$$
we define $\SW$ to be $f$ tensored with $\Q$ over  $\widehat{\Q}[G] = \Uhat(\Malcev(G))$.
By its definition, the isomorphism $\P$ is represented by $\SW$ at the chain level. 
Moreover, since the contracting homotopy $s$ is functorial by construction \cite{SW}, the chain maps $f$ and $\SW$ are functorial.
\end{proof}

To illustrate Theorem \ref{th:PSW}, let us consider the easy case where $G$ is a finitely generated free abelian group. 
In this case, we have a complete Hopf algebra isomorphism:
$$
\widehat{\operatorname{S}}(G \otimes \Q) \mathop{\longrightarrow}^\simeq \widehat{\Q}[G], \
g\otimes 1 \longmapsto \log(\gp{g})=\sum_{n\geq 1}\frac{(-1)^{n+1}}{n} \cdot (\gp{g}-1)^n.
$$
(Here, a same element of $G$ is denoted by $g$ or $\gp{g}$ depending on whether $G$ 
is regarded as a $\Z$-module or as a group.) 
Thus, the  abelian Lie algebra $\Malcev(G)$ can be identified with $G\otimes \Q$ (equipped with the trivial Lie bracket).
Theorem \ref{th:PSW} asserts that
$\Lambda^* (G\otimes \Q) = H_*(\Malcev(G);\Q)$ is isomorphic to $H_*(G;\Q)$, which is well-known. 
The isomorphism even exists with $\Z$-coefficients
and is defined at the chain level thanks to the Pontryagin product \cite[\S V.6]{Brown_groups}:
$$
\forall g_1,\dots,g_n \in G, \
g_1 \wedge \cdots \wedge g_n \longmapsto
\sum_{\sigma \in \mathfrak{S}_n} \varepsilon(\sigma) \cdot (\gp{g}_{\sigma(1)}\vert \cdots \vert \gp{g}_{\sigma(n)}).
$$
The canonical chain map $\SW$ happens to be a left-inverse of that map.

\begin{proposition}
\label{prop:abelian_case}
If $G$ is a finitely generated free abelian group,
then the chain map $\SW$ introduced in Theorem \ref{th:PSW} is given by
$$
\forall \gp{g}_1,\dots,\gp{g}_n \in G, \
\SW_n(\gp{g}_1\vert \cdots \vert \gp{g}_n) =
\frac{1}{n!} \cdot \log(\gp{g}_1) \wedge \cdots \wedge \log(\gp{g}_n).
$$
\end{proposition}

\begin{proof}
We use the same notation as in the proof of Theorem \ref{th:PSW}.
With the condition $\left[\Malcev(G),\Malcev(G)\right]=0$, 
the contracting homotopy (\ref{eq:contracting_homotopy}) is easily computed from its definition given in \cite{SW}.
We find 
$$
s(g_1 \cdots g_p \otimes h_1 \wedge \cdots \wedge h_q)
= \left\{\begin{array}{ll}
0 & \hbox{if } p=0\\
\frac{1}{p+q} {\displaystyle \sum_{i=1}^p g_1 \cdots \widehat{g_i} \cdots g_p \otimes g_i \wedge h_1 \wedge \cdots \wedge h_q} 
& \hbox{if } p>0
\end{array}\right.
$$
for all $p,q\geq 0$ and $g_1,\dots,g_p,h_1,\dots,h_q \in \Malcev(G)$. 
It can then be checked by induction on $n\geq 0$ that the chain map (\ref{eq:chain_map}) is given by the formula
$$
f_n(\gp{g}_1 \vert \cdots \vert \gp{g}_n) 
= \sum_{i_1,\dots,i_n \geq 1} 
\frac{1}{\prod_{k=1}^n (i_k-1)! \cdot \prod_{k=1}^n(i_k+\cdots +i_n)}
g_1^{i_1-1} \cdots g_n^{i_n-1} \otimes g_1 \wedge \cdots \wedge g_n
$$
for all  $\gp{g}_1,\dots,\gp{g}_n \in G$ and where  $g_1 := \log(\gp{g}_1),\dots,g_n := \log(\gp{g}_n) \in \Malcev(G)$.
By definition of the chain map $\SW$, 
$\SW_n(\gp{g}_1\vert \cdots \vert \gp{g}_n)$ is the term indexed by $i_1=\cdots=i_n=1$ in the above sum.
\end{proof}

\begin{corollary}
\label{cor:SW_1}
Let $G$ be a finitely generated torsion-free nilpotent group. Then the chain map $\SW$  introduced in Theorem \ref{th:PSW}
coincides with $\log: G \to \Malcev(G)$ in degree $1$.
\end{corollary}

\begin{proof}
According to Proposition \ref{prop:abelian_case}, this is true when $G$ is an infinite cyclic group.
The general case follows by functoriality.
\end{proof}

%
%
%
%
%
%

\bibliographystyle{abbrv}

\bibliography{EMH}

\def\cprime{$'$}
\begin{thebibliography}{10}

\bibitem{AC}
D.~Altsch{\"u}ler and A.~Coste.
\newblock Invariants of three-manifolds from finite group cohomology.
\newblock {\em J. Geom. Phys.}, 11(1-4):191--203, 1993.
\newblock Infinite-dimensional geometry in physics (Karpacz, 1992).

\bibitem{ABP}
J.~E. Andersen, A.~J. Bene, and R.~C. Penner.
\newblock Groupoid extensions of mapping class representations for bordered
  surfaces.
\newblock {\em Topology Appl.}, 156(17):2713--2725, 2009.

\bibitem{BKP}
A.~J. Bene, N.~Kawazumi, and R.~C. Penner.
\newblock Canonical extensions of the {J}ohnson homomorphisms to the {T}orelli
  groupoid.
\newblock {\em Adv. Math.}, 221(2):627--659, 2009.

\bibitem{Bourbaki}
N.~Bourbaki.
\newblock {\em \'{E}l\'ements de math\'ematique. {F}asc. {XXXVII}. {G}roupes et
  alg\`ebres de {L}ie. {C}hapitre {II}: {A}lg\`ebres de {L}ie libres.
  {C}hapitre {III}: {G}roupes de {L}ie}.
\newblock Hermann, Paris, 1972.
\newblock Actualit\'es Scientifiques et Industrielles, No. 1349.

\bibitem{Brown_groups}
K.~S. Brown.
\newblock {\em Cohomology of groups}, volume~87 of {\em Graduate Texts in
  Mathematics}.
\newblock Springer-Verlag, New York, 1982.

\bibitem{CE}
H.~Cartan and S.~Eilenberg.
\newblock {\em Homological algebra}.
\newblock Princeton University Press, Princeton, N. J., 1956.

\bibitem{Day1}
M.~B. Day.
\newblock Extending {J}ohnson's and {M}orita's homomorphisms to the mapping
  class group.
\newblock {\em Algebr. Geom. Topol.}, 7:1297--1326, 2007.

\bibitem{Day2}
M.~B. Day.
\newblock Extensions of {J}ohnson's and {M}orita's homomorphisms that map to
  finitely generated abelian groups.
\newblock Preprint \texttt{arXiv:0910.4777}, 2009.

\bibitem{DW}
R.~Dijkgraaf and E.~Witten.
\newblock Topological gauge theories and group cohomology.
\newblock {\em Comm. Math. Phys.}, 129(2):393--429, 1990.

\bibitem{Godin}
V.~Godin.
\newblock The unstable integral homology of the mapping class groups of a
  surface with boundary.
\newblock {\em Math. Ann.}, 337(1):15--60, 2007.

\bibitem{Harer}
J.~L. Harer.
\newblock The virtual cohomological dimension of the mapping class group of an
  orientable surface.
\newblock {\em Invent. Math.}, 84(1):157--176, 1986.

\bibitem{Heap}
A.~Heap.
\newblock Bordism invariants of the mapping class group.
\newblock {\em Topology}, 45(5):851--886, 2006.

\bibitem{Higgins}
P.~J. Higgins.
\newblock Presentations of groupoids, with applications to groups.
\newblock {\em Proc. Cambridge Philos. Soc.}, 60:7--20, 1964.

\bibitem{Igusa}
K.~Igusa.
\newblock {\em Higher {F}ranz-{R}eidemeister torsion}, volume~31 of {\em AMS/IP
  Studies in Advanced Mathematics}.
\newblock American Mathematical Society, Providence, RI, 2002.

\bibitem{IO}
K.~Igusa and K.~E. Orr.
\newblock Links, pictures and the homology of nilpotent groups.
\newblock {\em Topology}, 40(6):1125--1166, 2001.

\bibitem{JR}
W.~Jaco and J.~H. Rubinstein.
\newblock Layered-triangulations of 3-manifolds.
\newblock Preprint \texttt{arXiv:math/0603601}, 2006.

\bibitem{Jennings}
S.~A. Jennings.
\newblock The group ring of a class of infinite nilpotent groups.
\newblock {\em Canad. J. Math.}, 7:169--187, 1955.

\bibitem{Johnson_abelian}
D.~Johnson.
\newblock An abelian quotient of the mapping class group {$\mathcal{I}\sb{g}$}.
\newblock {\em Math. Ann.}, 249(3):225--242, 1980.

\bibitem{Johnson_survey}
D.~Johnson.
\newblock A survey of the {T}orelli group.
\newblock In {\em Low-dimensional topology ({S}an {F}rancisco, {C}alif.,
  1981)}, volume~20 of {\em Contemp. Math.}, pages 165--179. Amer. Math. Soc.,
  Providence, RI, 1983.

\bibitem{Kawazumi}
N.~Kawazumi.
\newblock Cohomological aspects of {M}agnus expansions.
\newblock Preprint \texttt{arXiv.org:math/0505497}, 2005.

\bibitem{Malcev_embedding}
A.~I. Mal{\cprime}cev.
\newblock Generalized nilpotent algebras and their associated groups.
\newblock {\em Mat. Sbornik N.S.}, 25(67):347--366, 1949.

\bibitem{Massuyeau}
G.~Massuyeau.
\newblock Infinitesimal {M}orita homomorphisms and the tree-level of the {LMO}
  invariant.
\newblock Preprint \texttt{arXiv:0809.4629}, 2008.

\bibitem{Morita_abelian}
S.~Morita.
\newblock Abelian quotients of subgroups of the mapping class group of
  surfaces.
\newblock {\em Duke Math. J.}, 70(3):699--726, 1993.

\bibitem{Morita_extension}
S.~Morita.
\newblock The extension of {J}ohnson's homomorphism from the {T}orelli group to
  the mapping class group.
\newblock {\em Invent. Math.}, 111(1):197--224, 1993.

\bibitem{Morita_linear}
S.~Morita.
\newblock A linear representation of the mapping class group of orientable
  surfaces and characteristic classes of surface bundles.
\newblock In {\em Topology and {T}eichm\"uller spaces ({K}atinkulta, 1995)},
  pages 159--186. World Sci. Publ., River Edge, NJ, 1996.

\bibitem{Morita_survey}
S.~Morita.
\newblock Structure of the mapping class groups of surfaces: a survey and a
  prospect.
\newblock In {\em Proceedings of the {K}irbyfest ({B}erkeley, {CA}, 1998)},
  volume~2 of {\em Geom. Topol. Monogr.}, pages 349--406 (electronic). Geom.
  Topol. Publ., Coventry, 1999.

\bibitem{MP}
S.~Morita and R.~C. Penner.
\newblock Torelli groups, extended {J}ohnson homomorphisms, and new cycles on
  the moduli space of curves.
\newblock {\em Math. Proc. Cambridge Philos. Soc.}, 144(3):651--671, 2008.

\bibitem{Penner_decorated}
R.~C. Penner.
\newblock The decorated {T}eichm\"uller space of punctured surfaces.
\newblock {\em Comm. Math. Phys.}, 113(2):299--339, 1987.

\bibitem{Penner_perturbative}
R.~C. Penner.
\newblock Perturbative series and the moduli space of {R}iemann surfaces.
\newblock {\em J. Differential Geom.}, 27(1):35--53, 1988.

\bibitem{Penner_bordered}
R.~C. Penner.
\newblock Decorated {T}eichm\"uller theory of bordered surfaces.
\newblock {\em Comm. Anal. Geom.}, 12(4):793--820, 2004.

\bibitem{Perron}
B.~Perron.
\newblock Homomorphic extensions of {J}ohnson homomorphisms via {F}ox calculus.
\newblock {\em Ann. Inst. Fourier (Grenoble)}, 54(4):1073--1106, 2004.

\bibitem{Pickel}
P.~F. Pickel.
\newblock Rational cohomology of nilpotent groups and {L}ie algebras.
\newblock {\em Comm. Algebra}, 6(4):409--419, 1978.

\bibitem{Quillen}
D.~Quillen.
\newblock Rational homotopy theory.
\newblock {\em Ann. of Math. (2)}, 90:205--295, 1969.

\bibitem{SW}
A.~A. Suslin and M.~Wodzicki.
\newblock Excision in algebraic {$K$}-theory.
\newblock {\em Ann. of Math. (2)}, 136(1):51--122, 1992.

\bibitem{Swarup}
G.~A. Swarup.
\newblock On a theorem of {C}. {B}. {T}homas.
\newblock {\em J. London Math. Soc. (2)}, 8:13--21, 1974.

\bibitem{Thomas}
C.~B. Thomas.
\newblock The oriented homotopy type of compact {$3$}-manifolds.
\newblock {\em Proc. London Math. Soc. (3)}, 19:31--44, 1969.

\bibitem{TV}
V.~G. Turaev and O.~Y. Viro.
\newblock State sum invariants of {$3$}-manifolds and quantum {$6j$}-symbols.
\newblock {\em Topology}, 31(4):865--902, 1992.

\end{thebibliography}

\end{document}